\documentclass[11pt, leqno]{amsart}

\makeatletter
\let\c@author\relax
\makeatother

\usepackage[backend=biber, citestyle=numeric-comp, bibstyle=trad-abbrv, url=false, doi=false, isbn=false]{biblatex}

\usepackage{fancyhdr}
\usepackage{marginnote}
\usepackage{todonotes}
\usepackage{amsmath}
\usepackage{hyperref}
\usepackage{comment}
\usepackage{soul}
\usepackage{mathtools}
\usepackage{amsfonts}
\usepackage{amsthm}
\usepackage{amssymb}
\usepackage{color}
\usepackage{comment}
\usepackage{enumerate}
\usepackage{dsfont}
\usepackage{geometry}
\usepackage{scalerel}[2014/03/10]
\usepackage[usestackEOL]{stackengine}
\newgeometry{tmargin=2.8cm, bmargin=3.5cm, lmargin=1.7cm, rmargin=1.7cm}
\mathtoolsset{showonlyrefs}

\definecolor{darkgreen}{rgb}{0.00, 0.50, 0.00}

\theoremstyle{plain}
\newtheorem{theo}{\bf Theorem}[section]
 
\newtheorem{coro}[theo]{\bf Corollary}
\newtheorem{lem}[theo]{\bf Lemma}

\newtheorem{prop}[theo]{\bf Proposition}

\theoremstyle{definition}
\newtheorem{ex}[theo]{\bf Example} 
\numberwithin{equation}{section}
\newtheorem{rem}[theo]{\bf Remark} 

\def\R{{\mathbb{R}}}

\def\N{{\mathbb{N}}}
\def\rn{{\mathbb{R}^{n}}}

\def\Rn{{\rn}}
\def\A{{A}}

\def\sob{{W^{1, A}}(\Omega)}

\def\sobzP{{V_0^{1, \Phi}(\Omega)}}

\def\sobB{{W^{1, B}}(\Omega)}

\def\vk{{\kappa}}
\def\d{{\,{\rm d}}}
\def\Cf{{E}}
\def\Df{{F}}
\def\Udel{{S_{\delta}}}
\def\UUU{{U}}

\def\avint{\,\ThisStyle{\ensurestackMath{%
  \stackinset{c}{.2\LMpt}{c}{.5\LMpt}{\SavedStyle-}{\SavedStyle\phantom{\int}}}%
  \setbox0=\hbox{$\SavedStyle\int\,$}\kern-\wd0}\int}

\title{Composition operators in Orlicz-Sobolev spaces}

\author {Michał Borowski \& Andrea Cianchi}
\address{Michał Borowski, Institute of Applied Mathematics and Mechanics,
University of Warsaw, ul. Banacha 2, 02-097 Warsaw, Poland}
\email{m.borowski@mimuw.edu.pl}

\address{Andrea Cianchi, Dipartimento di Matematica e Informatica \lq\lq U. Dini"\\
Universit\`a di Firenze,
Viale Morgagni 67/a,
50134 Firenze,
Italy} \email{andrea.cianchi@unifi.it}


\begin{document}

\begin{abstract} 
{The continuity of the Nemytskii operator between Orlicz-Sobolev spaces is investigated. Natural Orlicz-Sobolev versions of classical results for standard Sobolev are established. The results presented not only extend the latter, but also improve them 
 in borderline situations. Anisotropic   Orlicz-Sobolev spaces are included in our analysis. The results offered for this class of spaces are new even for customary anisotropic Sobolev spaces.}
\end{abstract}
\maketitle
\section{Introduction}
 Given an open set $\Omega \subset \rn$, $n \in\N$, and a function $f : \R \to \R$, the composition operator $T_f$ is defined as 
$$T_f(u)(x) = f(u(x)) \quad \text{for $x\in \Omega$,}$$
for a function $u : \Omega \to \R$. 

The nonlinear operator $T_f$ is also called the superposition operator or Nemytskii operator in the literature. Its action on various function spaces has been investigated, for instance,  in ~\cite{MarMiz, MarMiz2, BL, HardyThompson, BV, Bourbaud1, Bourbaud2, BCS, BMS, Moussai, Dinca}. In particular, information on the continuity of $T_f$ in Sobolev type spaces is of use in the theory of Partial Differential Equations and the Calculus of Variations.

A classical result from \cite{MarMiz2, MarMiz} asserts that, if $f$ is Lipschitz continuous and $|\Omega|<\infty$, then $T_f$ maps the Sobolev space $W^{1,p}(\Omega)$ into itself, and the operator
\begin{align}
    \label{jan21}
    T_f : W^{1,p}(\Omega) \to W^{1,p}(\Omega)
\end{align}
is continuous for every $p \in [1, \infty]$. Here, $|\Omega|$ denotes the Lebesgue measure of $\Omega$. The same conclusion holds if $|\Omega|=\infty$, provided that $f(0)=0$.
\\ A version of this result holds under the weaker assumption that $f$ is just locally Lipschitz continuous, and its derivative admits a power-type estimate.  However, in this case, only sufficiently regular domains $\Omega$ are admissible, and the operator
\begin{align}
    \label{jan22}
    T_f : W^{1,p}(\Omega) \to W^{1,q}(\Omega)
\end{align}
is continuous for some  range of exponents $q<p$, depending on $p, n$ and on the power which bounds the growth of $|f'|$. The latter result rests upon the Sobolev embedding theorem for the space $W^{1,p}(\Omega)$, which dictates the regularity required for $\Omega$. For the same reason, such regularity can be dispensed with when dealing with the homogeneous Sobolev space $V^{1,p}_0(\Omega)$ of functions vanishing on $\partial \Omega$, equipped with the gradient norm in $L^p(\Omega)$.

In the present paper, we are concerned with the continuity of the operator $T_f$ in the more general class of Orlicz-Sobolev spaces $W^{1,A}(\Omega)$. They are built on arbitrary Young functions $A: [0, \infty) \to [0, \infty]$, of which the powers $t^p$, for $p\geq 1$, are special instances. In analogy with~\eqref{jan21}, the continuity of the operator 
\begin{align}
    \label{jan23}
    T_f : W^{1,A}(\Omega) \to W^{1,A}(\Omega)
\end{align}
is established for any Young function $A$. This is our first result, which is the subject of Theorem~\ref{theo-conA}. A sharp Orlicz-Sobolev counterpart of \eqref{jan22} is also offered. Namely, under a bound for $|f'|$ in terms of a non-decreasing continuous function $\Cf$, the operator
\begin{align}
    \label{jan24}
    T_f : W^{1,A}(\Omega) \to W^{1,B}(\Omega)
\end{align}
is shown to be continuous for functions $B$ fulfilling a proper growth condition, depending on $A, \Cf$, and $n$.
This condition is prescribed by the optimal embedding theorem for Orlicz-Sobolev spaces  and takes a different form, depening  on whether $|\Omega|<\infty$ or $|\Omega|=\infty$. These two cases are dealt with in    Theorem~\ref{theo:conB1} and Theorem ~\ref{theo:conB2}, respectively. The regularity assumptions on $\Omega$ needed for \eqref{jan24} are the same as for \eqref{jan22}, and, as shown in Theorem \ref{theo:conB0}, can be dropped if  the spaces $W^{1,A}(\Omega)$ and $ W^{1,B}(\Omega)$ are replaced with their homogeneous versions $V^{1,A}_0(\Omega)$ and $ V^{1,B}_0(\Omega)$.

The continuity of $T_f$ between homogeneous anisotropic Orlicz-Sobolev spaces $V^{1,\Phi}_0(\Omega)$ is analyzed as well. They are defined via   $n$-dimensional Young functions $\Phi : \rn \to [0, \infty]$. Hence, the norm of a function in $ V^{1,\Phi}_0(\Omega)$ depends on the full gradient and not only on its length.  In Theorem~\ref{theo-conPhi}, the continuity of the operator
\begin{align}
    \label{jan25}
    T_f : V^{1,\Phi}_0(\Omega) \to V^{1,\Phi}_0(\Omega)
\end{align}
is proved for every $n$-dimensional Young function $\Phi$. Moreover, the operator 
\begin{align}
    \label{jan26}
    T_f : V^{1,\Phi}_0(\Omega) \to V^{1,\Psi}_0(\Omega)
\end{align}
is shown to be continuous under a bound on $\Psi$,  depending on $\Phi$, $n$ and the growth of $|f'|$, which extends the assumption required for \eqref{jan24}.  This is the content of Theorem~\ref{theo:conPhi0}.
\\
Importantly, our conclusions in the anisotropic setting are new, even for   classical orthotropic  functions of the form
$\Phi (\xi) = \sum _{i=1}^n |\xi_i|^{p_i}$,
where $p_i\in [1, \infty)$ for $i=1, \dots , n$ and $\xi= (\xi_1,  \dots , \xi_n)$. This special instance is discussed in Example~\ref{ex-ortho}. The more general class of functions given by 
$\Phi (\xi) = \sum _{i=1}^n A_i(|\xi_i|)$,
where $A_i$ is a Young function,  for $i=1, \dots , n$, is considered in
Corollary~\ref{coro:ortho}.

The lack of homogeneity of Young functions prevents some standard properties of Lebesgue and Sobolev spaces, such as some classical convergence theorems,  from holding in Orlicz and Orlicz-Sobolev spaces, unless additional assumptions are imposed on the defining Young functions. By contrast, the missing properties can typically be restored if the norm topology in the relevant spaces is replaced with the so-called modular topology. The latter is defined through integrals instead of norms and is not equivalent to the norm topology for Young functions that do not satisfy the $\Delta_2$-condition, a property which is not assumed in our main results. This is not a serious issue in the applications we alluded to above, where integrals of functions rather than their norms come into play.
 In the same spirit, we deal with the continuity of the operator $T_f$ in Orlicz-Sobolev spaces in the modular topology. This is not just a technical choice, inasmuch as the conclusions of our results can fail in the norm topology in the absence of the $\Delta_2$-condition, as demonstrated by Example~\ref{ex:counter}.

Let us mention that the continuity of composition of operators in Orlicz-Sobolev spaces was earlier addressed in the paper~\cite{HardyThompson}, whose results are improved and extended by our contribution in several respects. Primarily, the use of sharp Sobolev embedding theorems from \cite{ACPS, IsotropicSobolev, Cianchi_CPDE, AnisoSobolev} enables us to derive a natural optimal balance between the functions $A$ and $B$ in \eqref{jan24}, and $\Phi$ and $\Psi$ in \eqref{jan26}. This is not achieved in \cite{HardyThompson}, whose approach relies upon an older non-optimal (in general) Orlicz-Sobolev embedding of~\cite{DT}. Importantly, the results of \cite{HardyThompson} also require unnecessary extra hypotheses, including the $\Delta_2$-condition, on the Young functions defining the Orlicz-Sobolev spaces. Moreover, domains $\Omega$ with  infinite measure are excluded from the discussion in \cite{HardyThompson}, and anisotropic Orlicz-Sobolev spaces are not considered.

Our approach builds upon methods employed in \cite{BL} and~\cite{MarMiz} in the case of classical Sobolev spaces. However, dealing with  Orlicz-Sobolev spaces entails facing additional difficulties, including the lack of homogeneity and of the $\Delta_2$-property of the Young functions defining the spaces in question, the non-equivalence of the norm and the modular topologies, and the unconventional structure of  optimal Orlicz-Sobolev inequalities in integral form. To overcome these problems, specific arguments tailored for Orlicz spaces are developed in this work.

The paper is organized as follows. Section~\ref{sec:pre} is devoted to the necessary background on Young functions, Orlicz, Orlicz-Sobolev spaces, and their anisotropic versions. Our main results and their implementation in special instances are presented in the subsequent two sections. Specifically, Section~\ref{sec:iso} is devoted to 
 the isotropic regime, whereas the anisotropic realm is the subject of   Section~\ref{sec:aniso}. The proofs of the main results are accomplished in Section \ref{sec:proofs}, which is preceded by Section \ref{sec:tech} where some technical lemmas are collected.
\section{Background}\label{sec:pre}

Although the isotropic Orlicz-Sobolev spaces are a special case of their anisotropic generalizations, we provide a separate treatment for ease of presentation. The pertaining basic notions about Young functions and Orlicz spaces, and of their vectorial counterparts, are also recalled. We refer to the  monographs  \cite{HaHas, Krasn, RaoRen} for an exhaustive treatment of this topic.

\subsection{Young functions, Orlicz and Orlicz-Sobolev spaces}\label{sub-iso}

 A function $$A : [0, \infty) \to [0, \infty]$$ is said to be a \emph{Young function} if it is convex, left-continuous, satisfies $A(0)=0$, and is non-constant in $(0,\infty)$. As a consequence,
 one has that the function
 \begin{align}
     \label{incr}
     \frac{A(t)}t \quad \text{is non-decreasing in} \,\, (0, \infty).
 \end{align}
 Moreover,
 \begin{align}
     \label{Alambda}
     A(\lambda t) \leq \lambda A(t) \quad \text{if $\lambda \geq 1$ and $t\geq 0$.} 
 \end{align}
We shall say that the Young function $A$ is non-degenerate if 
\begin{align}
    \label{nondeg}
    A(t)>0 \,\,\,\text{for}\,\,\,t>0.
\end{align}
 A Young function $A$ satisfies the $\Delta_2$-condition (briefly, $A \in \Delta_2$) globally if there exists a constant $c\geq 1$ such that
 \begin{equation}\label{eq:ADelta-2}
 A(2t) \leq cA(t) \quad \text{for $t\geq 0$.}
  \end{equation}
 The function $A$
satisfies the $\Delta_2$-condition near infinity if it is finite-valued and~\eqref{eq:ADelta-2} is only fulfilled for sufficiently large $t$. Similarly, $A$ is said to satisfy the $\Delta_2$-condition near zero if~\eqref{eq:ADelta-2} is fulfilled only for sufficiently small $t \geq 0$.
{Notice that, if $A\in \Delta_2$ globally or near zero, then $A$ is non-degenerate.}
\\
Given any non-decreasing function $A : [0, \infty) \to [0, \infty]$, the generalized right-continuous inverse $A^{-1} : [0, \infty] \to [0, \infty]$ is defined as

\begin{equation}\label{A-inverse}
    A^{-1}(t) = \inf \{ s \geq 0 : A(s) > t \} \quad \text{for $t \geq 0$,}
\end{equation}
with the convention that $\inf \emptyset = \infty$. Notice that $t \leq A^{-1}(A(t))$ for $t \geq 0.$
\\
Two Young functions $A$ and $B$ are said to be equivalent globally/near infinity/near zero if there exists a positive constant $c$ such that
\begin{align}
    \label{2024-210}
    A(t/c) \leq B(t) \leq A(ct)
\end{align}
for $t\geq 0$/for $t\geq t_0$ and for some $t_0>0$/for $0\leq t\leq t_0$ and for some $t_0>0$, respectively. We shall write
\begin{align}
\label{2024-211} 
A \simeq B
\end{align}
to denote the  equivalence between $A$ and $B$ in the sense of \eqref{2024-210}. 
These notions will also be employed for nonnegative increasing functions, which are not necessarily Young functions.

Let $\Omega$ be a measurable set in $\rn$. We denote by $\mathcal M (\Omega)$ the set of measurable functions $u: \Omega \to \R$. The Orlicz space $L^A(\Omega)$ is defined as
$$ L^{A}(\Omega) = \bigg\{u\in \mathcal M(\Omega): \int_{\Omega} A\Big(\frac{|u(x)|}\lambda\Big)\d x < \infty \,\, \text{for some $\lambda >0$}\bigg\}\,,$$
and is equipped with the Luxemburg norm 
 \begin{equation*}
    \|u\|_{L^{A}(\Omega)} = \inf \bigg\{ \lambda > 0: \int_{\Omega} A\Big(\frac{|u(x)|}{\lambda} \Big)\d x \leq 1 \bigg\}\,.
\end{equation*}
We denote by $E^A(\Omega)$ the subspace of $L^{A}(\Omega)$ defined as
$$ E^{A}(\Omega) = \bigg\{u\in \mathcal M(\Omega): \int_{\Omega} A\Big(\frac{|u(x)|}\lambda\Big)\d x < \infty \,\, \text{for every $\lambda >0$}\bigg\}\,.$$
The  norms $\|\cdot \|_{L^A
(\Omega)}$ and $\|\cdot \|_{L^B (\Omega)}$ are equivalent if and only if either $|\Omega|<\infty$ and
$A$ and $B$ are equivalent near infinity, or $|\Omega|=\infty$ and $A$ and $B$ are equivalent globally.
\\ Assume that $|\Omega|<\infty$ and that either $p\geq 1$ and $\alpha \in \R$, or $p=1$ and $\alpha \geq 0$.  We denote by 
$L^p(\log L)^\alpha (\Omega)$ the Orlicz space, also called Zygmund space, associated with a Young function equivalent to $t^p(\log (1+t))^\alpha$ near infinity. Accordingly, the notation $L^p(\log \log  L)^\alpha (\Omega)$ stands for the Orlicz space associated with a Young function equivalent to $t^p\left(\log \log(1+t)\right)^{\alpha}$ near infinity. Also, given $\alpha >0$, we denote by $\exp L^\alpha (\Omega)$ the Orlicz space built upon a Young function equivalent to $e^{t^\alpha}$ near infinity.
\\
A sequence of functions $\{u_k\}$  in $L^{A}(\Omega)$ is said to converge modularly to a function $u \in L^{A}(\Omega)$ if
\begin{equation}\label{modular}
    \text{there exists $\lambda >0$ such that} \quad \lim_{k\to \infty}\int_{\Omega} A\left(\frac{|u_k(x) - u(x)|}{\lambda}\right)\d x= 0\,.
\end{equation}
We say that $u_k$ converges to $u$ modularly with constant $\lambda$ if the limit in \eqref{modular} holds for such a $\lambda$. 
\\  If $A$ is a non-degenerate Young function, modular convergence in $L^A(\Omega)$ implies convergence in measure. Moreover, if $|\Omega| < \infty$, then modular convergence in $L^A(\Omega)$ implies convergence in $L^1(\Omega)$.
\\
Recall that
\begin{equation}\label{eq:normconv-def}
    \lim_{k \to \infty}\|u_k - u\|_{L^A(\Omega)}=0 \quad \text{if and only if}  \quad \lim_{k\to \infty}\int_{\Omega} A\left(\frac{|u_k(x) - u(x)|}{\lambda}\right)\d x= 0   \quad \text{for every $\lambda >0\,$.}
\end{equation}
Thus, $u_k\to u$ in norm if and only if $u_k\to u$ modularly with every constant $\lambda >0$. In particular, norm and modular convergence in $L^A(\Omega)$ are equivalent if $A\in \Delta_2$.
\\
  Assume that $|\Omega|<\infty$ and $A$ and $B$ are Young functions equivalent near infinity. Then 
\begin{equation}
    \label{feb2}
   \text{ $u_k \to u$ \,\,  modularly in $L^A(\Omega)$ \,\, if and only if  \,\, $u_k \to u$ modularly in $L^B(\Omega)$.}
\end{equation}
This is a consequence of Vitali Convergence Theorem. 

Assume now that $\Omega$ is an open set. The Orlicz-Sobolev space $W^{1, A}(\Omega)$ is defined as
\begin{equation}\label{W1A}
    W^{1, A}(\Omega) = \{u \in L^{A}(\Omega) : \,\text{$u$ is weakly differentiable and $\nabla u \in L^{A}(\Omega)$}\},
\end{equation}
and is endowed with the norm given by
$$\|u\|_{W^{1,A}(\Omega)} = \|u\|_{L^{A}(\Omega)} + \|\nabla u\|_{L^{A}(\Omega)} \quad \text{for } u \in W^{1, A}(\Omega)\,.$$
\\
The alternate notation $W^{1}L^A(\Omega)$ instead of $W^{1,A}(\Omega)$ will also be employed when convenient.
\\ The homogeneous Orlicz-Sobolev space $V^{1, A}_0(\Omega)$ of those functions whose gradient belongs to $L^A(\Omega)$ and vanish on $\partial \Omega$ 
is suitably defined as follows.   Given any function $u: \Omega \to \mathbb R$, let $\widehat u : \rn \to \mathbb R$ be the function defined as 
 \begin{align}
     \label{extu}
     \widehat u(x) = \begin{cases}
         u(x) \quad & \text{if $x\in \Omega$}
        \\ 0 \quad & \text{if $x\in \rn \setminus \Omega$.}
     \end{cases}
 \end{align}
Then we set
\begin{align}\label{V1A0}
    V^{1, A}_0(\Omega) = \{u \in \mathcal M (\Omega):  \text{
    $\widehat u \in W^{1, 1}_{loc}(\Rn)$, $\nabla \widehat u \in L^{A}(\rn)$, $|\{ x\in \rn : |\widehat u(x)| > t\}| < \infty$ \, for $t>0$
\}}\,.
    \end{align}
    Clearly, $\widehat u=u$ in \eqref{V1A0} if $\Omega = \rn$.  The space $V^{1, A}_0(\Omega)$ is equipped with the norm given by 
    \begin{align*}
        \|u\|_{V^{1, A}_0(\Omega)}
= \|\nabla u\|_{L^A(\Omega)} \quad \text{for } u \in V^{1, A}_0(\Omega)\,.\end{align*}
Modular convergence in the spaces $W^{1, A}(\Omega)$ and $V^{1, A}_0(\Omega)$ is defined 
in analogy to \eqref{modular}.
The Sobolev type spaces $W^{1}E^A(\Omega)$ and $V^{1}_0E^A(\Omega)$  are defined as in \eqref{W1A} and \eqref{V1A0}, with the space $L^A(\Omega)$ replaced with $E^A(\Omega)$.
\\ { Let us notice that, if $|\Omega|<\infty$, then  thanks to a Poincar\'e type inequality in Orlicz spaces, $V^{1,A}_0(\Omega)\to W^{1, A}(\Omega)\to W^{1,1}(\Omega)$ for every Young function $A$.}
\\
{ 
Sobolev type inequalities and corresponding embeddings play a critical role in our results. They take a different form according to whether $n=1$ or $n \geq 2$.
\\ Consider first the case when $n=1$. Let $\Omega$ be an open set in $\R$ and let $A$ be a Young function. If $|\Omega|<\infty$, then
\begin{align}
    \label{1demb0}
\|u\|_{L^\infty(\Omega)} \leq |\Omega|A^{-1}\bigg(\frac 1{|\Omega|}\int_\Omega A(|u'|)\, \d x\bigg)
\end{align}
for $u \in V^{1, A}_0(\Omega)$.
\\ If $\Omega$ is any open set and $A(t) \simeq t$ near zero, then there exists a constant $c=c(A)$ such that
\begin{align}
    \label{1demb0inf}
\|u\|_{L^\infty(\Omega)} \leq c\int_\Omega A(|u'|)\, \d x
\end{align}
for $u \in V^{1, A}_0(\Omega)$.
\\ In both cases, 
\begin{align}
    \label{embV0R}
    V^{1, A}_0(\Omega) \to L^\infty (\R).
\end{align}
Assume that $\Omega \subset \R$ is an open interval and $A$ is any Young function. Then,
\begin{align}
    \label{embWR}
    W^{1, A}(\Omega) \to L^\infty (\R),
\end{align}
and 
\begin{align}
    \label{1demb}
    \|u\|_{L^\infty(\Omega)} \leq c \|u\|_{W^{1,A}(\Omega)}
\end{align}
for  $u\in W^{1,A}(\Omega)$, where  $c=c(A, \Omega)$ if $\Omega$ is bounded and $c=c(A)$ otherwise.
\\ The proof of inequalities \eqref{1demb0}, \eqref{1demb0inf}, and \eqref{1demb} is quite elementary. We provide a proof in Lemma \ref{lemma-1d}, Section \ref{sec:tech}, for completeness.}

Assume now that $n\geq 2$. Then optimal embeddings in Orlicz-Sobolev spaces can be described as follows.
\\
Let $A$ be a Young function such that
\begin{equation}\label{A-0}
\int_{0} \left( \frac{t}{A(t)}\right)^{1/(n-1)}\d t < \infty\,.
\end{equation}
The  Sobolev conjugate $A_n$ of $A$ is given by
\begin{align}
    \label{sobconj}
     A_n(t) = A(H_n^{-1}(t)) \quad \text{for $t \geq 0$\,,}
\end{align}
where
\begin{equation}\label{Hn}
    H_n(s) = \left( \int_{0}^{s} \left( \frac{t}{A(t)} \right)^{1/(n-1)} \d t\right)^{1/n'} \quad \text{for $s \geq 0$\,,}
\end{equation}
and $H_n^{-1}$ denotes the generalized left-continuous inverse of $H$.
\\
The Poincar\'e-Sobolev inequality of \cite[Theorem 3]{Cianchi_CPDE} (see also \cite{IsotropicSobolev} for an equivalent version) in Orlicz spaces tells us that 
\begin{align}
    \label{embV1A0}
    V^{1,A}_0(\Omega) \to L^{A_n}(\Omega)\,,
\end{align}
 and
\begin{align}
    \label{sonV1o}
    \int _\Omega A_n\Bigg(\frac {|u(y)|}{c (\int _\Omega
   A(|\nabla u|) \d x)^{1/n}}\Bigg)\d y \leq \int _\Omega A(|\nabla u|)\d x
\end{align}
for some constant $c=c(n)$ and for every  $u \in V^{1,A}_0(\Omega )$.
\\ When $|\Omega|<\infty$, the embedding \eqref{embV1A0}
continues to hold even if $A$ does not satisfy the assumption \eqref{A-0}, provided that $A_n$ is defined with $A$ replaced with a Young function equivalent to $A$ near infinity, which fulfills \eqref{A-0}. {In what follows, when \eqref{A-0} is not explicitly assumed, the function $A_n$ is defined via this replacement whenever needed, without further mentioning.}
\\  Notice that, if 
\begin{equation}\label{A-convinf}
    \int^{\infty} \left( \frac{t}{A(t)}\right)^{1/(n-1)}\d t < \infty\,,
\end{equation}
then 
$\lim_{s\to \infty}H_n(s)<\infty$,
and hence $A_n(t)=\infty$ for large $t$. Therefore, $V^{1, A}_0(\Omega)\to L^\infty(\Omega)$. On the other hand, this embedding fails in the complementary regime when
\begin{equation}\label{A-inf}
    \int^{\infty} \left( \frac{t}{A(t)}\right)^{1/(n-1)}\d t = \infty\,.
\end{equation}
{ Under the assumption \eqref{A-inf}, one also has that
\begin{align}
    \label{embE}
    V^{1,A}_0(\Omega) \to E^{A_n}(\Omega)\,,
\end{align}}
see \cite[Comments after Theorem 3]{Cianchi_CPDE}.

Let $\sigma \geq n$.  An open set $\Omega \subset \rn$  with $|\Omega|<\infty$ is said to satisfy a relative isoperimetric inequality with exponent $1/{\sigma '}$ if there exists a  positive constant $c$   such that 
\begin{equation}\label{isopineq}
c \min\{|G|, |\Omega \setminus G|\}^{\frac 1{\sigma '}} \leq P(G;\Omega)
\end{equation}
for every measurable set $G\subseteq \Omega$. Here, $P(G;\Omega)$ denotes the perimeter of $E$ relative to $\Omega$ in the sense of geometric measure theory. Recall that 
$$P(G;\Omega) = \mathcal H^{n-1} (\partial G \cap \Omega)$$
whenever $\partial G \cap \Omega$ is sufficiently smooth, where $\mathcal H^{n-1}$ stands for the $(n-1)$-dimensional Hausdorff measure. The assumption $\sigma \geq n$ is due to the fact that the inequality \eqref{isopineq} cannot hold if $\sigma < n$, whatever is $\Omega$. This can be shown by testing the inequality when $G$ is a ball and letting its radius tend to $0$.
\\
We denote by $\mathcal G _{1/{\sigma '}}$ the class of domains  in $\rn$ satisfying a relative isoperimetric inequality~\eqref{isopineq} with exponent $1/{\sigma'}$. These classes were introduced in~\cite{Ma0}, where membership of a domain $\Omega$ in  $\mathcal G _{1/{\sigma '}}$ was shown to be equivalent to    the validity of a Sobolev-Poincar\'e inequality between the norm of  functions in $L^{\sigma '}(\Omega)$
and the norm of their gradient in $L^1(\Omega)$. Notice that, in particular, every set $\Omega \in \mathcal G _{1/{\sigma '}}$ is connected. 
\\ Several customary families of domains  are contained in  $\mathcal G _{1/{n'}}$. This is the case of 
bounded Lipschitz domains, domains satisfying the cone condition, or John domains.

Let $\sigma \geq n$ and let $A_\sigma$ be the Young function defined as in~\eqref{sobconj}--\eqref{Hn}, with $n$ replaced with $\sigma$. Namely, 
\begin{align}
    \label{sobconj-sigma}
     A_\sigma(t) = A(H_\sigma^{-1}(t)) \quad \text{for $t \geq 0$\,,}
\end{align}
where
\begin{equation}\label{Hn-sigma}
    H_\sigma(s) = \left( \int_{0}^{s} \left( \frac{t}{A(t)} \right)^{1/(\sigma-1)} \d t\right)^{1/\sigma '} \quad \text{for $s \geq 0$\,.}
\end{equation}
Assume that  $\Omega \in \mathcal G _{1/{\sigma'}}$ for some $\sigma \geq n$. Then, 
\begin{align}
    \label{embW1A}
W^{1,A}(\Omega) \to L^{A_\sigma}(\Omega)\,,
\end{align}
 and,  by~\cite[Remarks 3.11 and 3.12]{SomeRes},  
\begin{equation}\label{eq:sobemb}
    \int_{\Omega} A_\sigma\left( \frac{|u - u_\Omega|}{c\left(\int_{\Omega} A(|\nabla u|)\d y\right)^{1/\sigma}} \right)\d x \leq \int_{\Omega} A(|\nabla u|)\d x 
\end{equation}
for some constant $c=c(\Omega)$ and 
for every $u \in W^{1, A}(\Omega)$. Here, 
$u_\Omega = \avint_{\Omega} u\d y$, the average of $u$ over $\Omega$.
\\ In particular, if
\begin{equation}\label{A-convinf-sigma}
    \int^{\infty} \left( \frac{t}{A(t)}\right)^{1/(\sigma-1)}\d t < \infty\,,
\end{equation}
then 
\begin{align}
    \label{w1Ainfinity}
W^{1, A}(\Omega)\to L^\infty(\Omega)\,.
\end{align}
On the other hand, if~\eqref{A-convinf-sigma} fails, then 
\begin{align}
    \label{embEsigma}
    W^{1,A}(\Omega) \to E^{A_\sigma}(\Omega)\,.
\end{align}
In analogy with~\eqref{embV1A0}, since we are assuming that $|\Omega|<\infty$, the embedding \eqref{embW1A} continues to hold with a suitable replacement of $A$ in the definition of $A_\sigma$, even if the integral in \eqref{Hn-sigma} does not converge. In what follows, such a replacement will implicitly be performed whenever needed in the definition of $A_\sigma$.

Let us now drop the assumption that $|\Omega|<\infty$, and suppose that $\Omega$ is just an extension domain in $\rn$, i.e., an open set admitting a linear operator $\mathcal E$ mapping functions defined in $\Omega$ to functions defined in $\rn$, such that
\begin{align*}
\mathcal E u(x) &= u(x) \quad \text{for $x \in \Omega\,$,}\\ 
\mathcal E : W^{1,1}(\Omega) \to W^{1,1}(\rn) &\quad \text{and} \quad  \mathcal E : W^{1,\infty}(\Omega) \to W^{1,\infty}(\rn)\,.
\end{align*}
Under these assumptions, one has that
\begin{align}\label{feb7}
    \mathcal E : W^{1,A}(\Omega) \to W^{1,A}(\rn)
\end{align}
for every Young function $A$ -- see \cite[Theorem 4.1]{Ci-Ran}. Since the norm of the operator \eqref{feb7} is independent of $A$, one also has that
\begin{align}
    \label{mod-ext}
    \int_{\rn} A(|\mathcal E u|) + A(|\nabla  \mathcal E u|)\d x \leq c \int_{\Omega} A(|u|) + A(|\nabla u|)\d x 
\end{align}
for every $u\in W^{1,A}(\Omega)$ making the right-hand side finite.\\
Such an extension operator is known to exist, for instance, for strongly Lipschitz domains in the sense of \cite[Chapter VI]{Stein-book} and in $(\varepsilon, \delta)$-domains defined in~\cite{Jones}. { On the other hand, a domain from the class $\mathcal G _{1/{n'}}$ need not be an extension domain.}
\\ 
Let 
$\widehat A_n$ be a Young function such that
\begin{align}
    \label{Anhat}
    \widehat A_n(t) \simeq \begin{cases}
        A_n(t) & \quad \text{near infinity}
        \\ A(t)  & \quad \text{near zero,}
    \end{cases}
\end{align}
where $A_n$ is defined with $A$ replaced, if necessary, with a Young function $\widehat A$ equivalent to $A$ near infinity and satisfying the condition \eqref{A-0}. 
Then, there exists a constant $c$
such that
\begin{align}
    \label{feb4}
    \|u\|_{L^{\widehat A_n}(\Omega)} \leq c \|u\|_{W^{1,A}(\Omega)} \quad \text{for } u \in W^{1,A}(\Omega)\,.
\end{align}
An equivalent form of this result is established in \cite[Theorem 2]{IsotropicSobolev}; the present version is contained in \cite[Theorem 3.5]{ACPS}. Notice that different replacements of $A$ near zero result in equivalent Young functions $A_n$ near infinity, and hence in Young functions $ \widehat A_n$, which are globally equivalent.
\\ In particular, under the assumption \eqref{A-convinf}, we have that
\begin{align}
    \label{feb4inf}
    \|u\|_{L^{\infty}(\Omega)} \leq c \|u\|_{W^{1,A}(\Omega)} \quad \text{for } u \in W^{1,A}(\Omega)\,.
\end{align}

\subsection{\texorpdfstring{$n$}{n}-dimensional Young functions, vector-valued Orlicz spaces, and anisotropic Orlicz-Sobolev spaces} \label{sub-aniso}

An \emph{$n$-dimensional Young function} is a function
 $$\Phi:\R^n\to [0, \infty]$$
 which is
 convex, even, lower semicontinuous,  finite in a neighborhood of $0$, and such that
\begin{align}\label{hpPhi}
   \text{ $\Phi(0)=0$ \,\, and \,\, $\lim_{|\xi|\to\infty}\Phi(\xi)=\infty\,$.}
\end{align}
The $n$-dimensional Young function $\Phi$ will be called non-degenerate if
\begin{align}
    \label{phinondeg}
    \Phi(\xi) > 0 \,\,\, \text{for} \,\,\,  \xi \neq 0.
\end{align}
Equivalence between $n$-dimensional Young functions is defined analogously to \eqref{2024-210}.
\\
Given a measurable set $\Omega \subset \rn$, we denote by 
 $L^\Phi (\Omega , \rn)$ the Orlicz space of measurable vector-valued functions $ \UUU : \Omega \to
\rn$ equipped with the norm
$$ \| \UUU \|_{L^\Phi (\Omega , \rn)}
 = \inf \left\{\lambda >0 :
 \int _\Omega \Phi \Big(\frac {1}{\lambda} \UUU (x)\Big) \d x \leq
 1\right\}.$$
The anisotropic homogeneous Orlicz-Sobolev space associated with $\Phi$ is defined as 
\begin{align}\label{V1Phi0}
    V^{1, \Phi}_0(\Omega) = \{u \in \mathcal M (\Omega):  \text{
    $\widehat u \in W^{1, 1}_{loc}(\Rn, \rn)$, $\nabla \widehat u \in L^{\Phi}(\rn)$, $|\{ x\in \rn : |\widehat u(x)| > t\}| < \infty$ \, for $t>0$
\}}\,,
    \end{align}
    where $\widehat u$ is given by \eqref{extu}.
    The space $V_0^{1, \Phi}(\Omega)$ is equipped with the norm
    \begin{equation*}
        \|u\|_{V_0^{1, \Phi}(\Omega)} = \|\nabla u\|_{L^{\Phi}(\Omega)}\,.
    \end{equation*}
Modular convergence in $L^\Phi (\Omega , \rn)$ and in   $ V^{1,\Phi}_0(\Omega)$ is defined analogously to \eqref{modular}. 
\\{Observe that, if $|\Omega|<\infty$, then $V^{1,\Phi}_0(\Omega)\to V^{1,1}_0(\Omega)\to W^{1,1}(\Omega)$ for every $n$-dimensional Young function $\Phi$.}
\\
Denote by
 $\Phi_\circ : [0, \infty ) \to [0, \infty)$  the Young function obeying
\begin{equation}\label{phistar}
|\{\xi \in \rn: \Phi_\circ (|\xi|) \leq t\}|  =|\{\xi \in \rn: \Phi
(\xi)\leq t\}|  \quad \hbox{for $t\geq 0\,$.}
\end{equation}
The function $\rn \ni \xi \mapsto \Phi _\circ(|\xi|)$   is a kind of    \lq\lq average in measure\rq\rq  \, of $\Phi$. 
\\
 Under the assumption   
\begin{equation}\label{A-0-aniso}
\int_{0} \left( \frac{t}{\Phi_\circ(t)}\right)^{1/(n-1)}\d t < \infty\,,
\end{equation}
the Sobolev conjugate $\Phi_n$ of $\Phi$ is defined as in~\eqref{sobconj}--\eqref{Hn}, with $A$ replaced with $\Phi_\circ$. Namely,
\begin{align}
    \label{sobconj-aniso}
     \Phi_n(t) = \Phi(H_\Phi^{-1}(t)) \quad \text{for $t \geq 0\,$,}
\end{align}
where
\begin{equation}
    H_\Phi(s) = \left( \int_{0}^{s} \left( \frac{t}{\Phi_\circ(t)} \right)^{1/(n-1)} \d t\right)^{1/n'} \quad \text{for $s \geq 0\,$,}
\end{equation}
  The anisotropic Sobolev embedding reads
\begin{align}
    \label{emb-aniso}
    V^{1, \Phi}_0(\Omega) \to L^{\Phi_n}(\Omega)\,,
\end{align}
and the corresponding
Poincar\'e-Sobolev inequality takes the form
\begin{equation}\label{eq:sobemb-2}
    \int_{\Omega} \Phi_n\left( \frac{|u|}{c\left(\int_{\Omega} \Phi(\nabla u)\d y\right)^{1/n}} \right)\d x \leq \int_{\Omega} \Phi(\nabla u)\d x 
\end{equation}
for some constant $c=c(n)$ and for every $u \in V^{1, \Phi}_0(\Omega)$, see~\cite{AnisoSobolev}.
\\  As in the isotropic framework, the embedding \eqref{emb-aniso} continues to hold, when $|\Omega|< \infty$, even if condition \eqref{A-0-aniso} fails, provided that $\Phi_n$ is defined with $\Phi$ suitably modified near zero. This modification will implicitly be adopted, whenever needed, without explicit mention throughout.
 In particular, the embedding 
\begin{align}
    \label{apr60}
    V^{1, \Phi}_0(\Omega) \to L^\infty(\Omega)
\end{align}
holds under the assumption
\begin{equation}\label{Phi-convinf}
    \int^{\infty} \left( \frac{t}{\Phi_\circ(t)}\right)^{1/(n-1)}\d t < \infty\,,
\end{equation}
whereas it fails if
\begin{equation}\label{Phi-inf}
 \int^{\infty} \left( \frac{t}{\Phi_\circ(t)}\right)^{1/(n-1)}\d t = \infty\,.
\end{equation}
Isotropic Orlicz and Orlicz-Sobolev spaces are recovered as special instances of their anisotropic counterparts with the choice 
$$\Phi (\xi) = A(|\xi|) \qquad  \text{for $\xi \in \R^n$}.$$
  Genuinely anisotropic instances of $n$-dimensional Young functions, sometimes called orthotropic in the literature, have the form
\begin{equation}\label{Phi=Ai}
\Phi (\xi) = \sum _{i=1}^n A_i(|\xi_i|) \qquad \hbox{for $\xi \in \R^n$,}
\end{equation}
where $A_i$ are Young functions and $\xi = (\xi_1, \dots \xi_n)$. 
A customary example of functions of this kind is
\begin{equation}\label{Phi=pi}
\Phi (\xi) = \sum _{i=1}^n |\xi_i|^{p_i} \qquad \hbox{for $\xi \in \R^n$,}
\end{equation}
where  $1\leq p_i<\infty$, for  $i=1,\dots,n$.
\\
The example \eqref{Phi=Ai} can be further generalized as
{\begin{equation}\label{Phi=full}
\Phi (\xi) =  \sum _{i=1}^n A_i\left(\sqrt{\sum_{j=1}^{n}\Big(\sum_{k=1}^n \alpha_{j, k}(i)\xi_k\Big)^2}\right)\qquad \hbox{for $\xi \in \rn$,}
\end{equation}}
where $A_i$ are Young functions, $n \in \mathbb N$, and, for each $i=1, \dots, n$, the matrix $(\alpha _{j, k}(i)) \in \mathbb R^{n\times n}$ is such that ${\rm det} (\alpha _{jk}(i))\neq 0$. A possible instance, for $n=2$, is
\begin{equation}
\label{trud}
\Phi (\xi) = |\xi_1 -\xi_2|^p + |\xi_1|^q\log (c+ |\xi _1|)^\alpha \quad \hbox{for $\xi \in \mathbb R^2$,}
\end{equation}
where either  $q>1$ and $\alpha \in \mathbb R$ , or $q=1$ and $\alpha \geq 0$, the exponent  $p\geq 1$, and $c$ is a sufficiently large constant for $\Phi$ to be convex. Another example   amounts to the function
\begin{equation}
\label{trud1}
\Phi (\xi) = |\xi_1 + 3 \xi_2|^p + e^{|2\xi_1-\xi_2|^\beta} -1  \quad \hbox{for $\xi \in \mathbb R^2$,  }
\end{equation}
with $p\geq 1$ and $\beta >1$.
\\ 
  Functions as in \eqref{Phi=full},
composed with another Young function,  result in $n$-dimensional Young functions as well.
One can also consider $n$-dimensional Young functions of the form
\begin{equation}\label{Phi=fuller}
\Phi (\xi) =  \sum _{i=1}^h B_i\Big( {\sum_{j=1}^{n}}A_{{i, j}}\Big(\Big|\sum_{k=1}^n \alpha_{{j, k}}(i)\xi_k\Big|\Big)\Big) \qquad \hbox{for $\xi \in \rn$,}
\end{equation}
where $B_i$ and $A_{ij}$ are Young functions, {$h \in \N$} and ${\rm det} ( \alpha_{{j, k}}(i) ) \neq 0$ {for $i=1,2,\dots, h$}, $j=1, \dots n$. An example is given by the function
$$\Phi (\xi)= (|\xi_1+\xi_2| + |\xi_2|^2)^3 + (|2\xi_3-\xi_4|^3 + |\xi_3+3\xi_4|^2)^2 \quad \hbox{for $\xi \in \mathbb R^4$.  }$$
Let us stress that there exist $n$-dimensional Young functions that do not split as in \eqref{Phi=fuller} -- see, e.g., \cite{ChNa}.

The Sobolev conjugate of orthotropic functions as in \eqref{Phi=Ai} takes, up to equivalence, a quite explicit form. Indeed, one has that
\begin{align}
    \label{ortho2}
    \Phi_\circ (t) \simeq     \overline A(t) \quad \text{for $t \geq 0$\,,}
\end{align}
where $\overline A$ is the Young function obeying
\begin{align}
    \label{ortho1}
    \overline A ^{\,-1}(t) = \Big(\prod _{i=1}^nA_i^{-1}(t)\Big)^{\frac 1n} \quad \text{for $t \geq 0\,$.}
\end{align}
Hence, if we denote by $\overline A_n$ the function defined as 
\begin{align}
    \label{Ahatn}
  \overline A_n (t) = \overline A (\overline H_n^{-1} (t)) \quad \text{for $t \geq 0\,$,}
\end{align}
where 
\begin{align}
    \label{Hhat}
    \overline H_n(s) = \left( \int_{0}^{s} \left( \frac{t}{\overline A(t)} \right)^{1/(n-1)} \d t\right)^{1/n'} \quad \text{for $s \geq 0\,$,}
\end{align}
 then
\begin{align}
    \label{ortho3}
    \Phi_n (t) \simeq \overline A_n(t) \quad \text{for $t \geq 0$.}
\end{align}

\section{Composition operators in isotropic Orlicz-Sobolev spaces}\label{sec:iso}
Our first result deals with the continuity of the composition operator $T_f$ from an Orlicz-Sobolev space into itself on open sets $\Omega \subset \rn$, with $n \geq 2$. It provides us with minimal assumptions on $f$ 
for the continuity of the map
\begin{align}
    \label{jan23bis}
    T_f : W^{1,A}(\Omega) \to W^{1,A}(\Omega)
\end{align}
in the modular topology.
{We point out that if  $\Omega$ is not assumed to have a finite measure, the results presented below hold, in particular, for $\Omega = \rn$.}
\begin{theo}\label{theo-conA}
Let $\Omega$ be an open set in $\rn$, with $n \geq 2$, and let $A$ be a non-degenerate Young function. Assume that $f : \R \to \R$ is a Lipschitz continuous function.  If $|\Omega| = \infty$, assume, in addition, that $f(0) = 0$. Then \eqref{jan23bis}  holds.
\end{theo}
 A counterpart of Theorem~\ref{theo-conA} for 
\begin{align}
    \label{jan23_0}
    T_f : V^{1,A}_0(\Omega) \to V^{1,A}_0(\Omega)
\end{align}
in the modular topology reads as follows.
\begin{theo}\label{coro-conA}
Let $\Omega$ be an open set in $\rn$, with $n\geq 2$, and let  $A$ be a  non-degenerate Young function. Assume that  $f : \R \to \R$ is such that $f(0) = 0$.{
\begin{enumerate}[(i)]
    \item   If $f$ is Lipschitz continuous, then \eqref{jan23_0} holds.
    \item 
    Assume, in addition, that the condition~\eqref{A-convinf} is in force,  and either $|\Omega|<\infty$, or $|\Omega| = \infty$ and the condition~\eqref{A-0} is also in force. Then \eqref{jan23_0} holds if $f$ is locally Lipschitz continuous.
\end{enumerate}}
\end{theo}
Note that Theorem~\ref{coro-conA} requires that $f(0) = 0$ even if $\Omega$ is of finite measure,  to guarantee that the function $f(u)$ {vanishes, in the suitable sense, on $\partial \Omega$, if $u$ does.}
The proof of Theorem \ref{coro-conA} will be skipped, inasmuch as it is a special case of Theorem \ref{theo-conPhi}, which deals with possibly anisotropic Orlicz-Sobolev spaces.

\medskip

Although Theorem  is still true, with the same proof,  when $n=1$ and $\Omega$ is an interval -- the case of interest in most applications -- its conclusion holds under the sole assumption that $f$ is locally Lipschitz continuous. A version of Part (i) of Theorem 
\ref{coro-conA} also holds under 
this weaker condition on $f$. As for Part (ii), a different condition is needed if $|\Omega|=\infty$.
These variants are due to the different form of the Poincar\'e-Sobolev inequalities  \eqref{1demb0} and \eqref{1demb} for $n=1$. 
\begin{theo}\label{1d}
Let $\Omega$ be an open interval in $\mathbb R$ and let $A$ be a  non-degenerate Young function. Assume that $f : \R \to \R$ is a locally Lipschitz continuous function.
\begin{enumerate}[(i)]
    \item If either $\Omega$ is bounded, or $\Omega$ is unbounded and $f(0)=0$, then \eqref{jan23bis} holds.
    \item Assume that 
$f(0) = 0$. If  either $\Omega$ is bounded, or $\Omega$ is unbounded and $A(t) \simeq t$ near zero, then ~\eqref{jan23_0} holds.
\end{enumerate}
\end{theo}

\begin{rem}  {\rm Under the assumption that $A\in \Delta_2$, the modular topology can be replaced with the norm topology in Theorems~\ref{theo-conA},~\ref{coro-conA}, and~\ref{1d}. Indeed, the two topologies are equivalent under this assumption. By contrast, the continuity of $T_f$ in \eqref{jan23bis} and \eqref{jan23_0} can fail in the norm topology if $A \notin \Delta_2$, as shown in Example~\ref{ex:counter}.
\\ On the other hand, the norm continuity can be restored for any Young function $A$ if the spaces  $W^{1, A}(\Omega)$ and $V^{1, A}_0(\Omega)$ are replaced with $W^{1}E^A(\Omega)$ and $V^{1}_0E^A(\Omega)$ in \eqref{jan23bis} and \eqref{jan23_0}. Namely, 
one has
 \begin{align}
    \label{jan23E}
    T_f : W^{1}E^A(\Omega) \to W^{1}E^A(\Omega)
\end{align}
and 
 \begin{align}
    \label{jan23_0E}
    T_f : V^{1}_0E^A(\Omega) \to V^{1}_0E^A(\Omega)
\end{align}
in the norm topology, for any Young function $A$, under the same assumptions on $\Omega$ and $f$ as in Theorems \ref{theo-conA}
and \ref{coro-conA},
respectively. This can be verified via a close inspection of the proofs of Theorems \ref{theo-conA}
and \ref{coro-conA}.
}
\end{rem}

We next focus on the continuity of $T_f$ in non-homogeneous Orlicz-Sobolev spaces when $f$ is not necessarily Lipschitz continuous but just locally Lipschitz with a derivative subject to some growth condition. In this case, one has to require regularity assumptions on the open set $\Omega$ under which a sharp Sobolev type embedding holds, allowing for different Orlicz-Sobolev domain and target spaces for $T_f$. Namely, we consider
 \begin{align}
    \label{jan51}
    T_f : W^{1,A}(\Omega) \to W^{1,B}(\Omega)
\end{align}
in the modular topology, where $f$ is a locally Lipschitz continuous function subject to the condition
 \begin{align}\label{jan50}
     |f'(t)| \leq \kappa \Cf(\kappa |t|) \quad \text{for a.e. $t\in \R$,}
 \end{align}
 for some non-decreasing, continuous function $\Cf : [0, \infty) \to [0, \infty)$ which is not identically equal to $0$, and 
 some constant $\kappa > 0$.  Note that the case when $E=0$ everywhere is trivial, as then $f$ is a constant, and hence $T_f$ agrees with the same constant. Here, and in what follows, $f'$ denotes a Borel representative of the derivative of $f$, which classically exists a.e. in $\mathbb R$. 

  The relation between the functions $A$, $B$, and $\Cf$  takes an easy form when $|\Omega|<\infty$. The assumptions in the general case are inevitably more articulate since they have to account for the behaviors of the functions $A$, $B$, and $\Cf$, not only near infinity but also near zero.
  \\ As already mentioned in Subsection \ref{sub-iso},  if the condition \eqref{A-0} is not explicitly assumed on $A$, then the function $A_n$ is defined with $A$ replaced, if necessary, with an equivalent Young function near infinity satisfying \eqref{A-0}. The conditions imposed on $A_n$ only involve large values of its argument and will not be affected by the specific replacement chosen for $A$ near zero. 

  The following theorem concerns  the property ~\eqref{jan51} when $|\Omega|<\infty$.
\begin{theo}\label{theo:conB1}

Assume that $n\geq 2$ and $\Omega \in \mathcal G_{1/n'}$. Let $A$ be a  non-degenerate Young function, let $B$ be a Young function,  and let $f : \R \to \R$ be a locally Lipschitz continuous function.
\begin{enumerate}[(i)]
    \item Assume that $A$ satisfies the condition \eqref{A-convinf}. Then \eqref{jan23bis} holds.
    \item Assume that $A$ satisfies the condition \eqref{A-inf}  and $f$ fulfills \eqref{jan50}. If there exists $t_0\geq 0$ such that
\begin{equation}\label{eq:inq-ass2}
B(t\Cf(H_n(t)))\leq A(t)
  \quad \text{ for $t \geq t_0,$}
        \end{equation}
        then  \eqref{jan51} holds. 
        Moreover, if  $u_k \to u$ modularly in  $W^{1, A}(\Omega)$ with constant $\lambda$, then $f(u_k) \to f(u)$ modularly in $W^{1, B}(\Omega)$  with constant $24\vk \max\{\lambda, \|u\|_{W^{1, A}(\Omega)}\}$.
\end{enumerate}
\end{theo}
{\begin{rem}
    \label{rem:domains}
    Clearly, Theorem \ref{theo:conB1} continues to hold for sets $\Omega$ which are the finite union of sets from the class $\mathcal G_{1/n'}$.
\end{rem}}

  A counterpart of Theorem~\ref{theo:conB1} for arbitrary extension domains reads as follows.
\begin{theo}\label{theo:conB2}
Assume that $\Omega$ is an extension domain in $\rn$, with $n\geq 2$.  Let $A$ be a  non-degenerate Young function, let $B$ be a Young function,  and let $f : \R \to \R$ be a locally Lipschitz continuous function such that $f(0)=0$.
\begin{enumerate}[(i)]
    \item   Assume that $A$ satisfies the condition \eqref{A-convinf}. Then \eqref{jan23bis} holds.
    \item Assume that $A$ satisfies the condition \eqref{A-inf} and $f$ fulfills  \eqref{jan50}. 
Suppose that~\eqref{eq:inq-ass2} holds and there exist $t_1\geq 0$ and a Young function $\Df$ such that
\begin{equation}\label{eq:inq-assD}
    B(t\Cf(\Df^{-1}(A(t))))\leq A(t)
        \quad \text{ for } \quad  \text{$0\leq t \leq t_1\,$.}
    \end{equation}
and
\begin{equation}\label{eq:ass-D1}
         \limsup_{t \to 0}\frac{\Df(\lambda t)}{A(t)} < \infty \quad \text{for every $\lambda >0\,$.}
    \end{equation}
    Then  \eqref{jan51} holds.
\end{enumerate}
\end{theo}
 \begin{rem}
The condition~\eqref{eq:inq-ass2} in Theorem~\ref{theo:conB1} and the condition~\eqref{eq:inq-assD} in Theorem ~\ref{theo:conB2} can be stated in an apparently more general formulation which allows for additional constants in the arguments of the functions $A, B,\Cf,\Df, H_n$. However, the content of the relevant theorems is not altered since replacements of Young functions with equivalent ones leave the associated Orlicz spaces unchanged.
\end{rem}

{If $\Cf$ is a (finite-valued) Young function instead of just a non-decreasing continuous function, in Part (ii) of Theorems \ref{theo:conB1} and \ref{theo:conB2}, then \eqref{jan51} also holds in the norm topology. In fact, one can even allow for the modular topology in $W^{1,A}(\Omega)$ and the norm topology in $W^{1,B}(\Omega)$.

\begin{prop}\label{rem:norm}
    Under the same hypotheses as in Part (ii) of Theorems \ref{theo:conB1} and \ref{theo:conB2}, assume, in addition, that $\Cf$ is a finite-valued Young function. Then \eqref{jan51} holds with $W^{1, A}(\Omega)$ and 
    $W^{1, B}(\Omega)$ endowed with the modular and the norm topology, respectively.
    \end{prop}}

If $A \in \Delta_2$ near zero, then the optimal choice for $\Df$ in Part (ii) of Theorem \ref{theo:conB2} is $A$. 
 Therefore, the following corollary holds.

\begin{coro}\label{coro:DoubleA}
 Assume that  $\Omega$ is an extension domain in $\rn$, with $n\geq 2$. Let $A$ and $B$ be Young functions. Assume that $A \in \Delta_2$ near zero and satisfies \eqref{A-inf}. Let $f : \R \to \R$ be a locally Lipschitz continuous function fulfilling the condition~\eqref{jan50} and such that $f(0)=0$.
Suppose that  \eqref{eq:inq-ass2} is satisfied  
    and
    \begin{equation}\label{eq:inq-assDbis}
      B(t\Cf(t)) \leq A(t)
        \quad \text{ for } \quad \text{$0\leq t \leq t_1$}
    \end{equation}
    for some $t_1>0$.
    Then   \eqref{jan51} holds.
     \end{coro}
On the other hand, one can show that the choice $\Df=A_n$  in  Part (ii) of Theorem~\ref{theo:conB2} is always admissible, although not optimal even if $A$ and $B$ are power functions and hence the corresponding spaces are classical Sobolev spaces.
\begin{coro}\label{coro:An}
  Let  $\Omega$ be an extension domain in $\rn$, with $n\geq 2$. Let $A$ and $B$ be Young functions. Assume that $A$ is non-degenerate and satisfies \eqref{A-inf}.
  Let $f : \R \to \R$ be a locally Lipschitz continuous function fulfilling the condition~\eqref{jan50} and such that $f(0)=0$.
Suppose that  \eqref{eq:inq-ass2} is satisfied with  $t_0 = 0$.
    Then   \eqref{jan51} holds.
     \end{coro}
\bigskip
A version of the results about \eqref{jan51} for homogeneous spaces, namely for the continuity of
 \begin{align}
    \label{jan71}
    T_f : V^{1,A}_0(\Omega) \to V^{1,B}_0(\Omega)
\end{align}
in the modular topology, takes a simpler form. This is the content of the next theorem. Only the case when the condition \eqref{A-inf} is in force is considered, as the complementary situation when \eqref{A-convinf} is fully described by Theorem \ref{coro-conA}.
\begin{theo}\label{theo:conB0}
 Let $\Omega$ be an open set in $\rn$, with $n\geq 2$. Let $A$ and $B$ be Young functions. Assume that $A$ is non-degenerate and satisfies the condition  \eqref{A-inf}.
 Let $f : \R \to \R$ be a locally Lipschitz continuous function fulfilling the condition~\eqref{jan50} and such that $f(0) = 0$.
 \begin{enumerate}[(i)]
     \item  Assume that   $|\Omega| < \infty$. If the inequality \eqref{eq:inq-ass2} is fulfilled for some $t_0\geq 0$, then \eqref{jan71} holds.
\item 
Assume that $A$ satisfies the condition \eqref{A-0}.  If the inequality \eqref{eq:inq-ass2} is fulfilled with $t_0= 0$, then \eqref{jan71} holds.
    \end{enumerate}
\end{theo}

 The proof of Theorem~\ref{theo:conB0} is omitted, since its content is  a special case of the anisotropic result proved in Theorem~\ref{theo:conPhi0} -- see Remark \ref{iso-aniso}.

  Theorem~\ref{theo:conB1} can be generalized to include less regular domains $\Omega$ from the class $\mathcal G_{1/\sigma'}$, for any $\sigma \geq n$. The proof of this extension makes use of the embedding $W^{1,A}(\Omega) \to L^{A_\sigma}(\Omega)$, with $\sigma \geq n$, and 
parallels that of \ref{theo:conB1}. The proof will be omitted for brevity.

\begin{theo}\label{theo:conB1sigma}

Assume that  $\Omega \in \mathcal G_{1/\sigma'}$ for some $\sigma \geq n \geq 2$.  Let $A$ be a  non-degenerate Young function, let $B$ be a Young function,  and let $f : \R \to \R$ be a locally Lipschitz continuous function.
\begin{enumerate}[(i)]
    \item  Assume that  the condition \eqref{A-convinf-sigma} is in force. Then \eqref{jan23bis} holds.
\item Assume that the condition \eqref{A-convinf-sigma} fails. Suppose that $f$ satisfies the condition~\eqref{jan50}.  If there exists $t_0\geq 0$ such that
\begin{equation}
B(t\Cf(H_\sigma(t)))\leq A(t)
        \quad \text{ for $t \geq t_0,$}
        \end{equation}
        then  \eqref{jan51} holds.
\end{enumerate}
\end{theo}
 The results of this section are now illustrated with a few examples. We begin with composition operators acting on classical Sobolev spaces.
\begin{ex}
    \label{ex:classical} Let $n\geq 2$ and assume that $\Omega \in \mathcal G_{1/n'}$. Let $f$
 be a locally Lipschitz function 
 such that 
\begin{align}\label{ex:zyg1}
|f'(t)| \leq \vk \Cf(\vk|t|) \quad \text{near infinity,}
\end{align}
 where $\vk > 0$ and $\Cf: [0, \infty) \to [0, \infty)$ is a continuous non-decreasing function.
\\
A classical result from~\cite[Theorem 2]{MarMiz} deals with the continuity of the operator
\begin{align}
    \label{classical1}
   T_f: W^{1,p}(\Omega) \to  W^{1,q}(\Omega).
\end{align}
 It asserts that \eqref{classical1} holds under one of the following circumstances:
\begin{enumerate}[(i)]
    \item  $1\leq p <n$ and
$$\text{$\Cf(t) =t^{r}$ for some $r>0$   and $1\leq q \leq \frac{np}{n + r(n - p)}$;}$$
\item $p=n$ and  $$ \Cf(t)= \begin{cases} 1 & \quad \text{for $q=n$;}
\\ t^{r} & \quad \text{for any $r>0$ and $1\le q<n$;}
\end{cases}$$
\item $p>n$ and $q=p$.
\end{enumerate}
These conclusions are recovered via Theorem \ref{theo:conB1} and sharpened in the borderline case when $p=n$. For instance, one has that
\begin{align}
    \label{classical2}
   T_f: W^{1,n}(\Omega) \to  \begin{cases}
       W^{1,q}(\Omega)& \quad  \text{for  $q<n$\quad if $\Cf(t) = \exp({t^{n'}})$;}
       \\ W^{1}L^n(\log L)^{-\alpha}(\Omega) &\quad \text{for  $\alpha \geq r(n-1)$ \quad if $\Cf(t)=  t^r$ for some $r>0$;}
       \\ W^{1}L^n(\log \log L)^{-\alpha}(\Omega)  & \quad
\text{for  $\alpha \geq rn$ \quad if $\Cf(t)=  \log^{r}(1+t)$ for some $r>0$.}       
   \end{cases}
\end{align}
{The result of~\cite[Theorem 2]{MarMiz} is also reproduced for extension domains $\Omega$, with possibly $|\Omega|=\infty$, via Corollary \ref{coro:DoubleA}.}
\end{ex}
The following example offers an analysis of composition operators between Zygmund-Sobolev spaces.
\begin{ex}\label{ex:Zygmund}
Let $n\geq 2$ and assume that $\Omega \in \mathcal G_{1/n'}$.  Let $f$ and $\Cf$ be as in Example \ref{ex:classical}.
 Assume that either 
 $p>1$ and $\alpha \in \R$, or  $p = 1$  and $\alpha \geq 0$, and that $q$ and $\beta$ satisfy analogous conditions. {Moreover, let $r>0$ and $\gamma \geq 0$.} Then, Theorem~\ref{theo:conB1} yields 
    \begin{equation}\label{eq:apr15-1}
        T_f: W^1L^p(\log L)^{\alpha}(\Omega) \to W^1L^q(\log L)^{\beta}(\Omega)
    \end{equation}
    in the following cases:
 {\begin{equation}\label{eq:exZygmund}
        \begin{cases}
        p < n, q < \frac{np}{n + r(n-p)} & \quad \Cf(t) = t^{r}\log^{\gamma}(1 + t);\\
        p < n, q = \frac{np}{n + r(n-p)},
           {\beta \leq n \frac{\alpha (1+r)-\gamma p}{n+r(n-p)}}
        &\quad  \Cf(t) = t^{r}\log^{\gamma}(1 + t);\\
        p = n, \alpha < n - 1, q < n & \quad \Cf(t) = \exp(t^{\frac{n}{n-1-\alpha}});\\
        p = q = n, \alpha < n - 1, \beta \leq \alpha(1 + r) - r(n-1) & \quad \Cf(t) = t^{r};
        \\
 p = n, \alpha = n - 1, q < n & \quad \Cf(t) = \exp(\exp(t^{n'})); \\
     p = q = n, \alpha = n - 1, \beta < n - 1 & \quad \Cf(t) = \exp(t^{n'});\\
       p = q, \alpha = \beta & \quad \text{$\Cf(t)=1$.}
        \end{cases}
    \end{equation}
        On the other hand, in the following cases:
        \begin{align}\label{apr25}
            \begin{cases}
            p=q=n, \alpha=\beta> n-1;
            \\ p=q>n, \alpha=\beta,
        \end{cases}
        \end{align}
        the continuity of the mapping ~\eqref{eq:apr15-1} holds upon the sole assumption that $f$ is locally Lipschitz continuous.}
\end{ex}
A variant of Example \ref{ex:Zygmund} for double-logarithmic Zygmund-Sobolev spaces is discussed in the following one.

\begin{ex}\label{ex:Zygmund2}
Let $n\geq 2$ and assume that $\Omega \in \mathcal G_{1/n'}$.  Let $f$ and $\Cf$ be as in Example \ref{ex:classical}.
Assume that $p,q,r, \alpha, \beta, \gamma$ are as in Example \eqref{ex:Zygmund}.
 Then, Theorem~\ref{theo:conB1} yields
    \begin{equation}\label{eq:apr15-2}
        T_f: W^1L^p(\log \log L)^{\alpha}(\Omega) \to W^1L^q(\log \log L)^{\beta}(\Omega)
    \end{equation}
    in the following cases:
    \begin{equation*}
    \begin{cases}
        p < n, q < \frac{np}{n + r(n-p)}  & \quad \Cf(t) = t^{r}(\log \log (e + t))^{\gamma};\\
        p < n, q = \frac{np}{n + r(n-p)}, 
        {\beta \leq n \frac{\alpha (1+r)-\gamma p}{n+r(n-p)}}
        & \quad \Cf(t) = t^{r}(\log \log (e + t))^{\gamma};\\
        p = n, q < n & \quad \Cf(t) = \exp(t^{n'}\log^{\frac{\alpha}{n-1}} (1 + t));\\
        p = q = n, \gamma  > 0, \beta \leq \alpha - n\gamma & \quad \Cf(t) = \log^{\gamma} (1 + t);\\
        p = q, \alpha = \beta & \quad \Cf(t) = 1.
    \end{cases}
    \end{equation*}
    {Moreover, if
    \begin{align}
        \label{apr26}
        p=q>n,\, \alpha =\beta,
    \end{align}
    then~\eqref{eq:apr15-2} holds upon the sole assumption that $f$ is locally Lipschitz continuous.}
\end{ex}

Exponential Sobolev spaces are considered as a further example.
\begin{ex}
    \label{ex-exp} Let $n\geq 2$ and 
     assume that $\Omega \in \mathcal{G}_{1/n'}$  and let $f$ be a locally Lipschitz function. Then, 
     \begin{equation}
         \label{ex-exp1}
         T_f : W^1\exp L^\alpha (\Omega)\to W^1\exp L^\alpha (\Omega)
     \end{equation}
     for every $\alpha >0$. This is a consequence of  Theorem~\ref{theo:conB1}, Part (i), as any Young function $A$ such that $A(t)\simeq e^{t^\alpha}$ near infinity satisfies the condition \eqref{A-convinf}.
\end{ex}

Our last example deals with 
 Sobolev spaces built upon  Young functions which do not satisfy the $\Delta_2$-condition near 0. This requires using the full strength of Theorem~\ref{theo:conB2} instead of its Corollary~\ref{coro:DoubleA} when dealing with arbitrary extension domains.
\begin{ex}
Let $\Omega$ be an extension domain in $\rn$, with $n\geq 2$. Let $f$ be locally Lipschitz continuous function such that  $f(0) = 0$ and
    $$|f'(t)| \leq \kappa |t|^r \quad \text{for $t \in \mathbb R$,}$$
    for some $r, \vk >0$. Assume  that $A$ and $B$ are Young functions such that $$A(t) \simeq  \exp(-t^{-\alpha}) \quad \text{and} \quad B(t) \simeq \exp(-t^{-\beta}) \quad \text{near zero,}$$ for  some $\alpha, \beta > 0$. Moreover, assume that $A$ and $B$ behave near infinity in such a way that the condition~\eqref{eq:inq-ass2} is fulfilled. Then, Theorem~\ref{theo:conB2} Part (ii) tells us that
    \begin{equation*}
        T_f : W^{1, A}(\Omega) \to W^{1, B}(\Omega)\,,
    \end{equation*}
   provided that $$\beta > \frac{\alpha}{r + 1}.$$
   This follows from an application of Theorem~\ref{theo:conB2} above with the choice
   $F(t) = \exp(-t^{-\gamma})$, with $\gamma > \alpha$ and sufficiently close to $\alpha$.
\end{ex}

{
We conclude this section with an example showing that the continuity of \eqref{jan23bis}  in the norm topology may fail
if $A\notin\Delta_2$.
\begin{ex}\label{ex:counter}
Assume that $\Omega = (0, 1)^n$, with $n\geq 1$,  $A(t) = t \exp t$, and $f(t) = \max(0, |t| - 1)$. Consider the function
$$u(x) = 1 +x_1(\log(x_1) - 1),\quad \text{where $x=(x_1, \dots , x_n)$}$$
and the sequence $\{u_k\}$ defined as 
\begin{equation*}
    u_k(x) = u(x) + \frac{1}{k}(\log(k) + 1)\,.
\end{equation*}
We claim that $u\in W^{1, A}(\Omega)$, $u_k \in W^{1, A}(\Omega)$ for $k \in \mathbb N$ and $\|u_k - u\|_{W^{1, A}(\Omega)} \xrightarrow{k \to \infty} 0$, but the sequence $\{f(u_k)\}$ does not converge to $f(u)$ in the norm topology of $W^{1, A}(\Omega)$.
\\
To verify these claims, observe that
 $\frac{\partial u}{\partial x_1}(x) = \log(x_1) < 0$, $\frac{\partial u}{\partial x_i} = 0$ for $i= 2, \dots n$, and $u(x) \in (0, 1)$. Clearly, $u \in L^A(\Omega)$, as $u$ is bounded. Also, $A(|\nabla u(x)|/2) = -\tfrac{1}{2}|x_1|^{-1/2}\log(|x_1|)$, whence  $\nabla u \in L^A(\Omega)$. Next, the function $u_k$ is bounded for $k \in \mathbb N$, and therefore $u_k \in L^A(\Omega)$. Moreover,  $\nabla u_k = \nabla u$,
and hence $u_k \in W^{1, A}(\Omega)$ as well. 
Since $u_k \to u$ uniformly in $\Omega$ and 
 $\nabla u_k - \nabla u = 0$, one has that $\|u_k - u\|_{W^{1, A}(\Omega)} \xrightarrow{k \to \infty} 0$.
\\
On the other hand, observe that for $x$ such that $x_1 < 1/k$, we have $u_k(x) > 1$ and $u(x) < 1$. Furthermore, $f'(t) = 1$ for $t > 1$ and $f'(t) = 0$ for $t \in (0, 1)$. Therefore, if $x_1 < 1/k$, then  
\begin{equation*}
    (f(u_k))'(x) - (f(u))'(x) = f'(u_k(x))\nabla u_k(x) - f'(u(x))\nabla u(x) = \nabla u_k(x) = \nabla u(x)\,.
\end{equation*}
Consequently,
\begin{equation*}
    \int_{\Omega} A(|f'(u_k)\nabla u_k - f'(u)\nabla u|)\d x \geq \int_{(0, 1/k) \times (0, 1)^{n-1}} A(|\nabla u(x)|) \d x = \int_{0}^{1/k} t^{-1}|\log(t)| = \infty\,.
\end{equation*}
This shows that the sequence $\{f(u_k)\}$ does not converge to $f(u)$ in norm topology of $W^{1, A}(\Omega)$.

\end{ex}}

\section{Composition operators in anisotropic Orlicz-Sobolev spaces}\label{sec:aniso}
In this section, we assume that $n \geq 2$ and provide anisotropic versions of the results stated in Section~\ref{sec:iso}. {Let us emphasize that whenever the domain $\Omega$ is not assumed to have a finite measure, the results of this section hold, in particular, for $\Omega = \rn$.}

We begin with the continuity of
\begin{align}
    \label{jan23_Phi}
    T_f : V^{1,\Phi}_0(\Omega) \to V^{1,\Phi}_0(\Omega)
\end{align}
in the modular topology, for an $n$-dimensional Young function $\Phi$. 
\begin{theo}\label{theo-conPhi}
Let $\Omega$ be an open set in $\rn$, with $n\geq 2$, and let  $\Phi$ be a  non-degenerate $n$-dimensional Young function. Assume that  $f : \R \to \R$ is such that $f(0) = 0$.
\begin{enumerate}[(i)]
    \item  If $f$ is Lipschitz continuous, then \eqref{jan23_Phi} holds.
    \item 
    Assume, in addition, that the condition~\eqref{Phi-convinf} is in force,  and either $|\Omega|<\infty$, or $|\Omega| = \infty$ and the condition~\eqref{A-0-aniso} is also in force. If $f$ is locally Lipschitz continuous, then \eqref{jan23_Phi} holds.
\end{enumerate}
\end{theo}
A counterpart of Theorem~\ref{theo:conB0} for $n$-dimensional Young functions is stated in the next theorem, 
dealing with the continuity of 
\begin{align}
    \label{jan72}
    T_f : V^{1,\Phi}_0(\Omega) \to V^{1,\Psi}_0(\Omega)
\end{align}
in the modular topology,  where $\Phi$ and $\Psi$ are possibly different $n$-dimensional Young functions and $f'$ is subject to a growth condition of the form \eqref{jan50}.
 Analogously to its isotropic version, Theorem~\ref{theo-conPhi} only deals with the case when \eqref{Phi-inf} holds, the complementary situation corresponding to the assumption \eqref{Phi-convinf} being considered in Theorem \ref{theo-conPhi}. 
 \\ 
As agreed in Subsection \ref{sub-aniso}, whenever the condition ~\eqref{A-0-aniso} is not explicitly assumed, the function  $\Phi_n$ is defined with $\Phi$ possibly modified near zero in such a way that ~\eqref{A-0-aniso} is fulfilled. The conditions to be imposed on the involved functions are 
independent of the choice of this replacement.
\\ Given an $n$-dimensional Young function $\Phi$ satisfying \eqref{A-0-aniso} and \eqref{Phi-inf}, let  $\vartheta : \rn \to [0, \infty)$ be the function implicitly defined 
 via the equation
    \begin{equation}\label{theta}
        \Phi_n(\vartheta(\xi)) = \Phi\left(\frac{\xi}{\Cf(\vartheta(\xi))}\right) \quad \text{ for $\xi \in \Rn$.}
    \end{equation}
 Note that the function $\vartheta$ is well defined, inasmuch as, under the current assumptions on $\Phi$, the function $\Phi_n$ is continuous and strictly increasing from $0$ to $\infty$, whereas the function $t \mapsto \Phi\big( \frac{\xi}{\Cf(t)} \big)$ is continuous and non-increasing for each fixed $\xi \in \rn$. Therefore, for every $\xi$, there exists a unique $t\geq 0$ such that $\Phi_n(t)-\Phi\big( \frac{\xi}{\Cf(t)} \big)=0$.
    
\begin{theo}\label{theo:conPhi0}
 Assume that $\Omega$ is an open set in $\rn$, with $n\geq 2$. Let $\Phi$ and $\Psi$ be $n$-dimensional Young functions.
Assume that $\Phi$ is non-degenerate and satisfies the condition \eqref{Phi-inf}.
 Let $\vartheta$ be the function defined by \eqref{theta}.
 Let $f : \R \to \R$ be a locally Lipschitz continuous function fulfilling the condition \eqref{jan50}  and such that $f(0) = 0$.
 \begin{enumerate}[(i)]
\item Assume that   $|\Omega| < \infty$. 
If there exists a constant $c > 0$ such that
\begin{equation}\label{eq:assumpt1}
        \Psi(\xi) \leq c + \Phi\left(\frac{\xi}{\Cf(\vartheta(\xi))}\right)
        \quad \text{ for $\xi \in \rn$,}
    \end{equation}
    then \eqref{jan72} holds.
\item Assume that $\Phi$ satisfies the condition \eqref{A-0-aniso}.  If \begin{equation}\label{eq:assumpt2}
        \Psi(\xi) \leq \Phi\left(\frac{\xi}{\Cf(\vartheta(\xi))}\right)
        \quad \text{ for $\xi \in \rn$,}
    \end{equation}
        then \eqref{jan72} holds.
\end{enumerate}
\end{theo}

{
\begin{rem}\label{iso-aniso}
   {The condition \eqref{eq:assumpt1}  recovers, {up to equivalence in the sense of \eqref{2024-211}}, the condition \eqref{eq:inq-ass2} in the case when $\Phi(\xi)=A(|\xi|)$
for some Young function $A$ fulfilling \eqref{A-0}    and \eqref{A-inf}, and $\Psi (\xi)= B(|\xi|)$ for some Young function $B$ -- see Lemma \ref{recover} in Section \ref{sec:tech}.}
\end{rem}
\medskip
{The following corollary concerns the continuity of \eqref{jan23_Phi} and \eqref{jan72} for orthotropic $n$-dimensional Young functions of the form \eqref{Phi=Ai} and locally Lipschitz continuous functions $f$. 
If $f$ is globally Lipschitz continuous,
the continuity of \eqref{jan23_Phi} is guaranteed by Theorem \ref{theo-conPhi}.}
{\begin{coro}\label{coro:ortho}
 Let $\Omega$ be an open set in $\rn$, with $n\geq 2$. Let $\Phi$ and $\Psi$ be $n$-dimensional Young functions of the form
    \begin{equation}\label{apr52}
        \Phi(\xi) = \sum_{i=1}^{n} A_i(|\xi_i|) \quad \text{and} \quad \Psi(\xi) = \sum_{i=1}^{n} B_i(|\xi_i|)\,,
    \end{equation}
 where $A_i$ are non-degenerate Young functions  and $B_i$ are Young functions, for  $i=1, \dots , n$.
Assume that $\Phi$ satisfies the condition \eqref{Phi-inf}, {namely
\begin{align}
    \label{apr50}
    \int^\infty\left( \frac{t}{\overline A(t)} \right)^{1/(n-1)} \d t=\infty\,,
\end{align}
where $\overline{A}$ is defined by \eqref{ortho1}.} 
Let $f : \R \to \R$ be a locally Lipschitz continuous function fulfilling the inequality \eqref{jan50}. Then, the following assertions hold.
\begin{enumerate}[(i)]
    \item Assume that   $|\Omega| < \infty$. If there exists $t_0 \geq  0$ such that 
\begin{equation}\label{eq:ass-ortho}
            B_i\big(A_i^{-1}\big(\overline A(t)\big)\Cf\big(\overline H_n(t)\big)\big) \leq \overline A(t)\quad \text{for  $t \geq t_0$,}
        \end{equation}
        for every $i=1, \dots , n$, then \eqref{jan72} holds.
\item Assume that $\Phi$ satisfies, in addition, the condition \eqref{A-0-aniso}. If \eqref{eq:ass-ortho} is fulfilled with $t_0=0$,
         then \eqref{jan72} holds.
\end{enumerate}
{In the case complementary to \eqref{apr50}, namely if 
\begin{align}
    \label{apr51}
    \int^\infty\left( \frac{t}{\overline A(t)} \right)^{1/(n-1)} \d t< \infty\,,
\end{align}
then \eqref{jan23_Phi} holds for every locally Lipschitz continuous function $f$, provided that either $|\Omega|<\infty$, or $|\Omega|=\infty$ and $\int_0\big( \frac{t}{\overline A(t)} \big)^{1/(n-1)} \d t< \infty.$}
\end{coro}
The special case of orthotropic power-type Sobolev spaces is analyzed below.
\begin{ex}
    \label{ex-ortho}
    Let $\Omega$ be an open set in $\rn$, with $n\geq 2$, such that $|\Omega| < \infty$ and let $f$ and $\Cf$ be as in Example \ref{ex:classical}. Given $p_i\geq 1$ and $q_i \geq 1$, for $i=1, \dots, n$, set 
    \begin{align}
        \label{ortho_1}
        \overrightarrow{p}= (p_1, \dots , p_n) \quad \text{and} \quad \overrightarrow{q}= (q_1, \dots , q_n).
    \end{align}
    Denote by $V^{1,\overrightarrow{p}}_0(\Omega)$  the orthotropic Orlicz-Sobolev spaces associated with the function $\Phi$ as in \eqref{Phi=pi}, with $p_i$ as in \eqref{ortho_1},
and let $V^{1,\overrightarrow{q}}_0(\Omega)$ be defined analogously. Let $\overline p$ be defined as
$$\frac{1}{\overline p}  = \frac{1}{n} \sum_{i=1}^{n} \frac{1}{p_{ i}}.$$
An application of Corollary \ref{coro:ortho} tells us that
\begin{align}
    \label{ortho_cont}
T_f: V^{1,\overrightarrow{p}}_0(\Omega) \to  V^{1,\overrightarrow{q}}_0(\Omega)
\end{align}
if:
\begin{enumerate}[(i)]
    \item $1\leq \overline p < n$ and 
\begin{equation*}
    \text{$\Cf(t)=t^{r}$ for some $r  {\geq} 0$ and  
$q_i \leq \frac{\overline pnp_i}{n\overline p + p_ir \left(n
 - \overline p\right)}$ \,\,for $i=1, \dots , n$},
\end{equation*}
\item $\overline p =n$ and 
$$\text{$\Cf(t)=  \exp({t^{n'}})$  
and  $q_i<p_i$ \,\,for $i=1, \dots , n$,}$$
\item $\overline p >n$ and $$\text{$q_i  = p_i$ \,\,for $i=1, \dots , n$.}$$
\end{enumerate}
Different choices of the function $E$ yield \begin{align}
    \label{ortho_cont-bis}
T_f: V^{1,\overrightarrow{p}}_0(\Omega) \to  V^{1,\Psi}_0(\Omega)
\end{align}
with $\Psi$ non-necessarily of power type. For instance, set $\overrightarrow{\alpha} = (\alpha_1, \dots , \alpha_n)$ and denote by  $ V_0^{1}L^{\overrightarrow{p}}(\log L)^{\overrightarrow{\alpha}}(\Omega)$ the anisotropic Orlicz-Sobolev space built upon the $n$-dimensional Young function
$$\Psi (\xi) = \sum_{i=1}^n |\xi|^{p_i}\log^{\alpha_i} (c+|\xi|)$$
for sufficiently large $c$. Then, in the borderline case (ii), namely when $\overline p =n$, one has 
\begin{align}
   T_f: V_0^{1, \overrightarrow{p}}(\Omega) \to  \begin{cases}
       V_0^{1}L^{\overrightarrow{p}}(\log L)^{-\overrightarrow{\alpha}}(\Omega)& \quad  \text{for $\alpha_i \geq \frac{rp_i}{n'}$\quad if $\Cf(t) = t^r$ for some $r > 0$}\\
       V_0^{1}L^{\overrightarrow{p}}(\log \log L)^{-\overrightarrow{\alpha}}(\Omega)& \quad  \text{for $\alpha_i \geq rp_i$\quad if $\Cf(t) = \log^r(1+t)$ for some $r > 0$.}
   \end{cases}
\end{align}

\end{ex}

 \section{Technical lemmas}\label{sec:tech}

\subsection{Isotropic spaces}

We begin with some properties of Young functions related to the assumptions of our main results.
\begin{lem}\label{lem:lem1}
    Let $n\geq 2$, let $A$ and $B$ be Young functions, and let   $\Cf : [0, \infty) \to [0, \infty)$ be a
    non-decreasing function.
    Assume that $A$ fulfills the condition~\eqref{A-inf}   and the inequality \eqref{eq:inq-ass2}
holds.
    Then, there exists a constant $c$ such that 
    \begin{equation}\label{eq:lem1}
        B\left(\Cf(s)t/2\right) \leq c + A_n(s) + A(t) \quad \text{for $t,s >0$.}
    \end{equation}
\end{lem}
\begin{proof}
As the condition~\eqref{A-inf} ensures that $A$ is finite valued, the inequality \eqref{eq:inq-ass2} implies that $B$ is also finite valued and
\begin{equation}\label{eq:inq-ass}
      A^{-1}(t) \Cf(A_n^{-1}(t)) \leq B^{-1}(t)\quad \text{for $t \geq A(t_0)$.} 
    \end{equation}
    Therefore,
\begin{equation}\label{eq:inq-ass'}
      A^{-1}(t) \Cf(A_n^{-1}(t)) \leq B^{-1}(t) + c'\quad \text{for $t \geq 0$,} 
    \end{equation}    
where $c'= t_0\Cf(A_n^{-1}(A(t_0)))$.
The inequality \eqref{eq:lem1} can be deduced from \eqref{eq:inq-ass'} via the following adaptation of an argument from 
the proof of~\cite[Theorem 2.1]{ONeil}. Let $s, t\geq 0$. If $A_n(s)\geq A(t)$, then
\begin{align}
    \label{apr5}
E(s)t \leq E(A_n^{-1}(A_n(s)))A^{-1}(A(t)) \leq E(A_n^{-1}(A_n(s)))A^{-1}(A_n(s)) \leq B^{-1}(A_n(s))+c',
\end{align}
whence
\begin{align}\label{apr6}
    B(E(s)t/2) \leq B((B^{-1} (A_n(s))+c')/2)\leq A_n(s)+ B(c').
\end{align}
If $A(t)\geq A_n(s)$, then
\begin{align}
    \label{apr7}
E(s)t \leq E(A_n^{-1}(A_n(s)))A^{-1}(A(t)) \leq E(A_n^{-1}(A(t)))A^{-1}(A(t)) \leq B^{-1}(A(t))+c',
\end{align}
whence
\begin{align}\label{apr8}
    B(E(s)t/2) \leq B((B^{-1} (A(t))+c)/2) \leq A(t)+B(c').
\end{align}
The inequality \eqref{eq:lem1} follows from \eqref{apr6} and \eqref{apr8} with $c=B(c')$.
\end{proof}
An analogous (and simpler) argument as in the proof of Lemma \ref{lem:lem1} yields the following result.
\begin{lem}\label{lem:inqD}
    Let $A$, $B$, and $\Df$ be Young functions, and let $\Cf : [0, \infty) \to [0, \infty)$ be a
    non-decreasing function.   Assume that  $A$ and $F$ are finite-valued, and
    \begin{equation}\label{eq:inq-assD2}
    B(t\Cf(\Df^{-1}(A(t))))\leq A(t)
        \quad \text{ for } \quad  \text{$ t \geq 0$.}
    \end{equation}
    Then
    \begin{equation}\label{eq:consinq}
    B(\Cf(s)t) \leq \Df(s) + A(t) \quad \text{ for} \,\,\,s, t \geq 0. 
\end{equation}
\end{lem}
}
 
\begin{lem}
    \label{LemmaF=An}
Let $n\geq 2$ and let $A$ be a Young function satisfying the condition \eqref{A-0}. Then 
\begin{align}
    \label{apr1}
        \limsup_{t \to 0}\frac{A_n(\lambda t)}{A(t)} < \infty \quad \text{for every $\lambda >0$.}
\end{align}
\end{lem}
\begin{proof}
By the property \eqref{incr},
$$\frac {A(t)}t \leq \frac {A(1)}1 \quad \text{for $t\in (0,1)$.}$$
Hence, there exists a constant $c>0$ such
 that
$$H_n(t) \geq c t^{\frac 1{n'}} \quad \text{for $t\in (0,1)$.}$$
Therefore,
$$H_n^{-1}(t) \leq (t/c)^{n'} \quad \text{near zero.}$$
Fixing $\lambda >0$, we have that $\lambda t \leq 1$ near zero. Thus,
$$H_n^{-1}(\lambda t) \leq \lambda^{n'} t^{\frac 1{n-1}} t \leq t \quad \text{near zero.}$$
Thus,
$$\frac{A_n(\lambda t)}{A(t)} = \frac{A(H_n^{-1}(\lambda t))}{A(t)}\leq \frac{A(t)}{A(t)}=1 \quad \text{near zero,}$$
whence \eqref{apr1} follows.
\end{proof}

The following results concern the convergence properties of sequences in Orlicz and Orlicz-Sobolev spaces.

{\color{black}
\begin{lem}
    \label{int-conv}
    Let $\Omega$ be a measurable subset of $\rn$, with $n\geq 1$, and let $A$ be a finite-valued  non-degenerate Young function. Assume that the sequence $\{u_k\} \subset L^A(\Omega)$ and the function $u\in L^A(\Omega)$ satisfy
    $$\int_\Omega A\Big(\frac{|u_k-u|}{\lambda_0}\Big)\d x \xrightarrow{k \to \infty} 0$$
    for some $\lambda_0>0$. {Let $\lambda_1>0$ be such that
    $$\int_\Omega A\Big(\frac{|u|}{\lambda_1}\Big)\d x<\infty.$$}
 Then,  
    \begin{align}
        \label{int-conv1}
        \int_\Omega A\Big(\frac{|u_k|}\lambda\Big)\d x \xrightarrow{k \to \infty}  \int_\Omega A\Big(\frac{|u|}\lambda\Big)\d x
    \end{align}
    for every $\lambda \geq 2\max(\lambda_0, \lambda_1\}$. 
    \\ In particular, if $\{u_k\} \subset E^{A}(\Omega)$, $u \in E^{A}(\Omega)$ and $\|u_k - u\|_{L^{A}(\Omega)}\to 0$, then~\eqref{int-conv1} holds for every $\lambda > 0$.
\end{lem}
\begin{proof}
 First, assume that $|\Omega| < \infty$. By Vitali Convergence Theorem, the sequence 
    $\Big\{A\Big(\frac{|u_k-u|}{ \lambda_0}\Big)\Big\}$ is uniformly integrable and converges to $0$ in measure as $k \to \infty$. By the convexity of $A$, 
    \begin{align*}
A\Big(\frac{|u_k|}\lambda\Big)  \leq   A\Big(\frac{2|u_k-u|}{ \lambda}\Big) +  A\Big(\frac{2|u|}{\lambda}\Big) \quad \text{for $\lambda \geq 2\max\{\lambda_0, \lambda_1\}$.}
    \end{align*}
Hence, the sequence  $\Big\{A\Big(\frac{|u_k|}{ \lambda}\Big)\Big\}$ is also uniformly integrable. 
\\ On the other hand, the convergence in measure of $\Big\{A\Big(\frac{|u_k-u|}{ \lambda_0}\Big)\Big\}$ to $0$ implies that $u_k \to u$ in measure. In turn, this implies that the sequence $\Big\{A\Big(\frac{|u_k|}{ \lambda}\Big)\Big\}$ converges in measure to $\Big\{A\Big(\frac{|u|}{ \lambda}\Big)\Big\}$ for $\lambda >0$. The limit \eqref{int-conv1} follows from an application of Vitali Convergence Theorem again.\\
{\color{black} Now,   assume that $\Omega$ is an arbitrary open set. Given $R > 0$, denote by $B_R$ the ball centered at $0$, with radius $R$. An application of ~\eqref{int-conv1} to $\Omega \cap B(0, R)$ enables us to deduce that
\begin{align*}
    &\limsup_{k \to \infty} \left| \int_{\Omega} A \left( \frac{|u_k|}{\lambda} \right) \d x - \int_{\Omega} A \left( \frac{|u|}{\lambda} \right) \d x \right|
    \leq \limsup_{k \to \infty} \left|\int_{\Omega \setminus B(0, R)} A \left( \frac{|u_k|}{\lambda} \right) \d x - \int_{\Omega \setminus B(0, R)} A \left( \frac{|u|}{\lambda} \right) \d x \right|\\
    &\leq \limsup_{k \to \infty} \int_{\Omega \setminus B(0, R)} A \left( \frac{2|u_k -  u|}{\lambda} \right) \d x + 2\int_{\Omega \setminus B(0, R)} A \left( \frac{2|u|}{\lambda} \right) \d x
    = 2\int_{\Omega \setminus B(0, R)} A \left( \frac{2|u|}{\lambda} \right) \d x.
\end{align*}
Hence, we get~\eqref{int-conv1} by passing to the limit as $R\to \infty$.
\\ The assertion concerning sequences in $L^E(\Omega)$
 is  a consequence of~\eqref{int-conv1} and of property~\eqref{eq:normconv-def}.
}
\end{proof}}

\begin{lem}\label{lem:convEAn}
   Assume that $n\geq 2$ and $\Omega \in \mathcal G_{1/n'}$. Let $A$ be a  non-degenerate Young function satisfying the condition \eqref{A-inf}.
   Assume that the sequence  $\{u_k\} \subset W^{1, A}(\Omega)$  and the function $u \in  W^{1, A}(\Omega)$ are such that $u_k \to u$ modularly in $W^{1, A}(\Omega)$.
    Then ${\|u_k - u\|_{L^{A_n}(\Omega)} \xrightarrow{k \to \infty} 0}$.
\end{lem}
\begin{proof}
    Let $\lambda_c > 0$ be such that
    \begin{equation*}
        \int_{\Omega} A\left( \frac{|u_k - u|}{\lambda_c}\right)\d x \xrightarrow{k \to \infty} 0 \quad \text{and} \quad \int_{\Omega} A\left( \frac{|\nabla u_k - \nabla u|}{\lambda_c}\right)\d x \xrightarrow{k \to \infty} 0\,. 
    \end{equation*}
Owing to the property \eqref{eq:normconv-def}, the conclusion will follow if we show that
    \begin{equation}\label{eq:goal-lem2}
        \int_{\Omega} A_n\left(\frac{|u - u_k|}{\lambda}\right)\d x \xrightarrow{k \to \infty} 0 \quad \text{for every $\lambda >0$.}
    \end{equation}
    Fix any $\lambda >0$. Let
     $c$ be the constant appearing in ~\eqref{eq:sobemb}. By Vitali convergence Theorem,   there exists $\delta > 0$ such that 
    \begin{equation}\label{eq:deltaA}
        \sup_{k} \int_{G} A\left(\frac{|\nabla u - \nabla u_k|}{\lambda_c}\right)\d x \leq \left(\frac{\lambda}{4c\lambda_c}\right)^n,
    \end{equation}
    for every measurable set $G \subset \Omega$ such that $|G| < \delta$.
    Note that $u_k \xrightarrow{k \to \infty} u$ in $L^1(\Omega)$ and in measure. Therefore, {given any $s > 0$},  we have that $|\{|u - u_k| \geq s\}| < \delta$ for sufficiently large $k$. Hence, by~\eqref{eq:deltaA},
    \begin{equation}\label{eq:deltaA2}
        \int_{\{|u - u_k| \geq s\}} A\left( \frac{|\nabla u - \nabla u_k|}{\lambda_c}\right)\d x \leq \left(\frac{\lambda}{4c\lambda_c}\right)^n\,.
    \end{equation}
    Now, observe that
    \begin{equation*}
        \int_{\{|u - u_k| < 2s\}} A_n\left(\frac{|u - u_k|}{\lambda}\right)\d x \leq 
        \int_{\Omega} A_n\left(\frac{\min(2s, |u-u_k|)}{\lambda}\right)\d x\,,
    \end{equation*}
    and the integrand in the last integral can be bounded by $A_n(2s/\lambda)$, independently of $k$. Thus, by Lebesgue Dominated Convergence Theorem,  the right-hand side of the last inequality tends to $0$. Consequently,
    \begin{equation}\label{eq:goal-part1}
        \int_{\{|u - u_k| < 2s\}} A_n\left(\frac{|u - u_k|}{\lambda}\right)\d x \xrightarrow{k \to \infty} 0\,.
    \end{equation}
    Now, define the function $g_s : \R \to \R$ as follows
    \begin{equation}\label{gs}
        g_s(t) = \begin{cases}
            t + s & \quad \text{if $t < -s;$}\\
            0 & \quad \text{if $-s \leq t \leq s;$}\\
            t - s & \quad \text{if $s < t$.}
        \end{cases}
    \end{equation}
    Since $g_s$ is Lipschitz continuous, a standard result in the theory of Sobolev spaces ensures that for any weakly differentiable function $v$ in $\Omega$, one has 
    \begin{equation*}
       |g_s(v)| = \max(0, |v| - s)  \, \quad \text{and} \quad \,\nabla (g_s(v)) = {\chi_{\{|v| \geq s\}}} \nabla v \quad \text{a.e. in $\Omega$.}
    \end{equation*}
    Here, $\chi_Z$ denotes the characteristic function of a set $Z\subset \rn$.
   The following chain holds:
    \begin{align}\label{eq:goal-part2-1}
        &\int_{\{|u - u_k| \geq 2s\}} A_n\left(\frac{|u - u_k|}{\lambda}\right)\d x\\ &\leq \int_{\{|u - u_k| \geq 2s\}} A_n\left(\frac{2(|u - u_k| - s)}{\lambda}\right)\d x \leq \int_{\Omega} A_n\left(\frac{2|g_s(u - u_k)|}{\lambda}\right)\d x \nonumber\\
        &\leq \int_{\Omega} A_n\left(\frac{4|g_s(u - u_k) - \avint_{\Omega} g_s(u - u_k)\d y|}{\lambda}\right)\d x + \int_{\Omega}A_n\left(\frac{4\left|\avint_{\Omega} g_s(u - u_k)\d y\right|}{\lambda}\right)\d x\,.\nonumber
    \end{align}
    Observe that
    \begin{equation*}
        \left| \avint_{\Omega} g_s(u - u_k)\d y \right| \leq \frac{\|u - u_k\|_{L^1(\Omega)}}{|\Omega|} \xrightarrow{k \to \infty} 0\,.
    \end{equation*}
    Therefore,
    \begin{equation}\label{eq:help-dec1}
        \int_{\Omega}A_n\left(\frac{4\left|\avint_{\Omega} g_s(u - u_k)\d y\right|}{\lambda}\right)\d x \xrightarrow{k \to \infty} 0\,.
    \end{equation}
    By~\eqref{eq:deltaA2}, 
    \begin{equation*}
    \int_{\Omega} A\left( \frac{|\nabla(g_s (u - u_k))|}{\lambda_c}\right)\d x = \int_{\{|u - u_k| \geq s\}} A\left( \frac{|\nabla u - \nabla u_k|}{\lambda_c}\right)\d x \leq \left(\frac{\lambda}{4c\lambda_c}\right)^n\,,
    \end{equation*}
   for sufficiently large $k$. 
    Hence, via ~\eqref{eq:sobemb},  we infer that
    \begin{align}\label{eq:mar28-6}
        \int_{\Omega} A_n&\left(\frac{4|g_s(u - u_k) - \avint_{\Omega} g_s(u - u_k)\d y|}{\lambda}\right)\d x \leq \int_{\Omega} A_n\left(\frac{|g_s(u - u_k) - \avint_{\Omega} g_s(u - u_k)\d y|}{c\lambda_c\left(\int_{\Omega} A\left( \frac{|\nabla(g_s (u - u_k))|}{\lambda_c}\right)\d y\right)^{1/n}}\right)\d x\\
        &\leq \int_{\Omega} A\left(\frac{|\nabla(g_s (u-u_k))|}{\lambda_c}\right)\d x \leq \int_{\Omega} A\left(\frac{|\nabla u - \nabla u_k|}{\lambda_c}\right)\d x \xrightarrow{k \to \infty} 0\,.\nonumber
    \end{align}
   Combining equations \eqref{eq:mar28-6}, ~\eqref{eq:goal-part2-1} and~\eqref{eq:help-dec1} implies that
    \begin{equation}\label{apr10}
        \int_{\{|u - u_k| \geq 2s\}} A_n\left(\frac{|u - u_k|}{\lambda}\right)\d x \xrightarrow{k \to \infty} 0\,.
    \end{equation}
    Equation~\eqref{eq:goal-lem2} follows from~\eqref{eq:goal-part1} and ~\eqref{apr10}.
\end{proof}
\begin{lem}\label{lem:convED}
    Let $\Omega$ be an extension domain in $\rn$, with $n\geq 2$, and let $A$ and $\Df$ be Young functions.  Assume that $A$ is non-degenerate and satisfies the condition~\eqref{A-inf},
    \begin{equation}\label{eq:ass-D2}
        \limsup_{t \to 0}\frac{\Df(t/\lambda)}{A(t)} < \infty \quad \text{for every $\lambda >0$\,\,  and there exists $t_1 > 0$ such that\,\,} \Df(t) \leq { A_n(t)} \text{ for $t \geq t_1$.} 
    \end{equation}
    Then,
    \begin{align}
        \label{mar10}
        W^{1, A}  (\Omega) \to E^{\Df}(\Omega)\,.
    \end{align}
    Moreover, 
    if $u \in W^{1, A}(\Omega)$ and $\{u_k\}\subset W^{1, A}(\Omega)$ are such that $u_k \to u$  modularly in $W^{1, A}(\Omega)$, then $$\|u - u_k\|_{L^\Df(\Omega)} \xrightarrow{k \to \infty} 0\,,$$ and 
    \begin{equation}\label{feb30}
   \int_\Omega F\Big(\frac{|u_k|}\lambda\Big) \d x \xrightarrow{k \to \infty} \int_\Omega F\Big(\frac{|u|}\lambda\Big) \d x 
    \end{equation}
    for every $\lambda >0$.
\end{lem}
\begin{proof} Since $\Omega$ is an extension domain, there exists a linear extension operator $\mathcal{E}$ satisfying the property \eqref{mod-ext}. Let us still denote by $u$ and $u_k$
the functions $\mathcal{E}u$ and $\mathcal{E} u_k$. Thanks to \eqref{mod-ext}, we have that $u_k\to u$ modularly in $W^{1, A}(\rn)$. Pick any $\lambda_c > 0$ such that 
\begin{align}
    \label{mar1}
    \int_{\rn} A(|u_k - u|/\lambda_c) \xrightarrow{k \to \infty} 0 \quad \text{and} \quad \int_{\rn} A(|\nabla u_k - \nabla u|/\lambda_c) \xrightarrow{k \to \infty} 0.
\end{align}
 Let us fix $\lambda > 0$. Note that the limsup in ~\eqref{eq:ass-D2} implies the existence of a constant $c_d$ such that  $\Df(t/\lambda) \leq c_dA(t/\lambda_c)$ if $t/\lambda \leq t_1$. Therefore,
    \begin{equation}\label{eq:help2}
        \int_{\{|u - u_k| \leq t_1\lambda\}} \Df\left( \frac{|u - u_k|}{\lambda}\right)\d x \leq c_d\int_{\{|u - u_k| \leq t_1\lambda\}} A\left( \frac{|u - u_k|}{\lambda_c}\right)\d x \xrightarrow{k \to \infty} 0\,.
    \end{equation}
Now, let $\widehat A$ and $\widehat A_n$ be as in \eqref{Anhat}. {By replacing, if necessary, $F$ with an equivalent Young function, we may assume that
 $F(t) \leq \widehat A_n(t)$ for $t \geq t_1$.} Set $s =t_1\lambda/2$ and let $g_s$ be the function defined as in   \eqref{gs}.
    Observe that
    \begin{align}\label{eq:help23-10-3}
        &\int_{\{|u - u_k| \geq t_1\lambda\}} \Df\left( \frac{|u - u_k|}{\lambda}\right)\d x \leq \int_{\{|u - u_k| \geq t_1\lambda\}} \widehat A_n\left( \frac{|u - u_k|}{\lambda}\right)\d x \\ \nonumber &\leq \int_{\{|u - u_k| \geq t_1\lambda\}} \widehat A_n\left( \frac{2(|u - u_k|-s)}{\lambda}\right)\d x \leq \int_{\rn} \widehat A_n\left( \frac{2|g_s(u - u_k)|}{\lambda}\right)\d x.
    \end{align}
    Choose $t_0 > 0$ such that $A(t) = \widehat A(t)$ for all $t \geq t_0$. Thereby, $\widehat A(t) \leq A(t) + A(t_0)$ for  $t \geq 0$. Hence,
    \begin{align}\label{eq:help23-10}
        &\int_{\rn} \widehat A \left( \frac{|\nabla g_s(u - u_k)|}{\lambda_c} \right)\d x = \int_{\{|u - u_k| \geq s\}} \widehat A \left( \frac{|\nabla g_s(u - u_k)|}{\lambda_c} \right)\d x \\ \nonumber 
        &\leq \int_{\{|u - u_k| \geq s\}} A(t_0)\d x + \int_{\{|u - u_k| \geq s\}} A \left( \frac{|\nabla g_s(u - u_k)|}{\lambda_c} \right)\d x \\ \nonumber 
        &\leq |\{|u - u_k| \geq s\}|A(t_0) + \int_{\rn} A \left(\frac{|\nabla u - \nabla u_k|}{\lambda_c}\right)\d x \xrightarrow{k \to \infty} 0\,,
    \end{align}
    where the last limit follows from convergence of $u_k$ to $u$ in measure and \eqref{mar1}. Owing to~\eqref{eq:help23-10},  one has  that
    \begin{equation}\label{eq:help23-10-2}
        \int_{\rn} \widehat A \left( \frac{|\nabla g_s(u - u_k)|}{\lambda_c} \right)\d x \leq \left( \frac{\lambda}{2C\lambda_c} \right)^n
    \end{equation}
    for sufficiently large $k$,
    where  $c$ is the constant appearing in~\eqref{sonV1o}. Coupling the inequality~\eqref{sonV1o} with~\eqref{eq:help23-10-2} yields
    \begin{align*}
    \int_{\rn} \widehat A_n\left( \frac{2|g_s(u - u_k)|}{\lambda}\right)\d x &\leq \int_{\rn} \widehat A_n\left( \frac{|g_s(u - u_k)|}{C\lambda_c \left( \int_{\rn} \widehat A(|\nabla g_s(u - u_k)|/\lambda_c) \right)^{1/n}}\right)\d x\\
    &\leq  \int_{\rn} \widehat A \left( \frac{|\nabla g_s(u - u_k)|}{\lambda_c} \right)\d x \xrightarrow{k \to \infty} 0,
    \end{align*}
    whence, by~\eqref{eq:help23-10-3}, we deduce that
    \begin{equation*}
        \int_{\{|u - u_k| \geq t_1\lambda\}} \Df\left( \frac{|u - u_k|}{\lambda}\right)\d x \xrightarrow{k \to \infty} 0\,.
    \end{equation*}
    The latter limit combined  with~\eqref{eq:help2} tells us that $\int_{\rn} \Df\left( \frac{|u - u_k|}{\lambda}\right)\d x \xrightarrow{k \to \infty} 0$.    Hence, 
    \begin{align}
        \label{mar2}
        \int_{\Omega} \Df\left( \frac{|u - u_k|}{\lambda}\right)\d x \xrightarrow{k \to \infty} 0 \quad \text{for every $\lambda >0$.}
    \end{align}
   By the arbitrariness of $\lambda$, the property \eqref{eq:normconv-def} ensures that
    $\|u - u_k\|_{L^\Df(\Omega)}\to 0$.
    \\ The limit \eqref{feb30} now follows from \eqref{mar2}, via
    Lemma \ref{int-conv}.
    \\  Finally, consider any $u \in W^{1, A}(\Omega)$ and $\lambda > 0$, and set $u_k = \min\{k, \max\{-k, u\}\}$ for $k \in \N$. By equation ~\eqref{mar2}, there exists $k_0$ such that $\int_{\Omega} F(2|u-u_{k_0}|/\lambda) < \infty$. Moreover, the assumption~\eqref{eq:ass-D2} ensures that there exists a constant $c = c(k_0, \lambda, u) > 0$ such that  $F(2t/\lambda) \leq cA(t/\|u\|_{L^A(\Omega)})$ for $t \leq k_0$. Therefore
    \begin{align*}
        \int_{\Omega} F \left( \frac{|u|}{\lambda} \right)\d x &\leq \int_{\Omega} F \left(\frac{2|u-u_{k_0}|}{\lambda} \right)\d x + \int_{\Omega} F \left( \frac{2\min(|u|, k_0)}{\lambda} \right)\\
        &\leq \int_{\Omega} F \left(\frac{2|u-u_{k_0}|}{\lambda} \right)\d x + c\int_{\Omega} A \left( \frac{|u|}{\|u\|_{L^A(\Omega)}} \right) < \infty\,.
    \end{align*}
    By the arbitrariness of $\lambda$, the embedding ~\eqref{mar10} follows.
\end{proof}
We conclude this subsection with a proof of the inequalities \eqref{1demb0}, \eqref{1demb0inf}, and \eqref{1demb}.

{
\begin{lem}
    \label{lemma-1d}  Let $\Omega$ be an open set in $\R$ and let $A$ be a Young function. 
    \begin{enumerate}[i)]
        \item 
        Assume that $|\Omega|<\infty$. 
       Then, the inequality \eqref{1demb0} holds.
        \item
          Assume that $A(t) \simeq t$ near $0$. Then, the inequality \eqref{1demb0inf} holds.
        \item Assume that $\Omega$ is an interval. Then, the inequality~\eqref{1demb} holds.
        \end{enumerate}
\end{lem}}
\begin{proof} 
{ Part $(i)$. Given $u\in V^{1, A}_0(\Omega)$, we have that
    \begin{align}
        \label{apr15}
        |u(x)| \leq \int_\Omega |u'|\d y \quad \text{for $x\in \Omega\,$.}
    \end{align}
   Thus, by Jensen's inequality,
    \begin{align}
        \label{apr16}
        A\bigg(\frac{|u(x)|}{|\Omega|}\bigg) \leq A\bigg(\frac{1}{|\Omega|}\int_\Omega |u'|\d y \bigg) \leq \frac{1}{|\Omega|}\int_\Omega A(|u'|)\d y 
        \quad \text{for $x\in \Omega$,}
    \end{align}
    whence \eqref{1demb0} follows.
\\ Part $(ii)$. Our assumption on $A$ ensures that there exists a constant $c$ such that 
 $c A(t) \geq t$ for  $t \geq 0$. Therefore,
    \begin{equation*}
         |u(x)| \leq \int_\Omega |u'|\, \d y \leq
         c\int_\Omega A(|u'|)\, \d y
         \quad  \text{for $x\in \Omega$,}
    \end{equation*}
    namely \eqref{1demb0inf}.}
\\  Part $(iii)$. First, assume that $\Omega$  is bounded.  Let $u$ be such that $\|u\|_{W^{1, A}(\Omega)} \leq 1$. A counterpart of~\eqref{apr15} and Jensen's inequality again imply that
     \begin{align}
        \label{apr17}
        |u(x)| & \leq \inf |u| + \int_\Omega |u'|\d y \leq \frac{1}{|\Omega|}\int_\Omega |u|\d y + \int_\Omega |u'|\d y 
        \\
        & \nonumber \leq A^{-1}\bigg(\frac{1}{|\Omega|}\int_\Omega A(|u|)\d y\bigg) + |\Omega|A^{-1}\bigg(\frac{1}{|\Omega|}\int_\Omega A(|u'|)\d y\bigg)\\
        &\leq A^{-1}(1/|\Omega|) + |\Omega|A^{-1}(1/|\Omega|)   \quad  \text{for $x\in \Omega\,$.}
    \end{align}
   An application  of the inequality ~\eqref{apr17} to the function $\frac{u}{\|u\|_{W^{1, A}(\Omega)}}$ yields \eqref{1demb} for an arbitrary function $u \in W^{1, A}(\Omega)$. 
    \\ Next, if $\Omega$ is unbounded, then it is either $\R$ or a half-line. Assume first that $\Omega =\mathbb R$.
     Suppose, for the time being, that 
    \begin{align}
        \label{apr23}
        A(1)=1\,.
    \end{align} Hence, $A(t)\geq t$ for $t\geq 1$.
    Let $u\in W^{1,A}(\mathbb R)$ be such that $\|u\|_{W^{1,A}(\mathbb R)} = 1$. Denote by $g$ the function defined as in \eqref{gs}, with $s=1$. Set $S = \{|u|>1\}$. Then, $S$ is an open set with $|S|< \infty$ and 
$g(u) \in V^{1, A}_0(S)$. From the inequality \eqref{1demb0} we deduce that
\begin{align}
    \label{apr19}
\|g(u)\|_{L^\infty(\Omega)} \leq |S|A^{-1}\bigg(\frac 1{|S|}\int_S A(|g(u)'|)\d x\bigg)  \leq |S|A^{-1}\bigg(\frac 1{|S|}\int_\Omega A(|u'|)\d x\bigg)\,. 
\end{align}
On the other hand,
\begin{align}
    \label{apr20}
    |S| \leq \int_S |u|\d y \leq \int_S A(|u|)\d y  \leq 1.
\end{align}
Owing to the property \eqref{incr}, we deduce from \eqref{apr19} and \eqref{apr20} that
\begin{align}
    \label{apr21}
\|g(u)\|_{L^\infty(\Omega)} \leq A^{-1}\bigg( \int_\Omega A(|u'|)\d x\bigg) \leq A^{-1}(1)=1.
\end{align}
Hence,
\begin{align}
    \label{apr22}
    \|u\|_{L^\infty(\mathbb R)}\leq \|(|u|-1)\chi_S\|_{L^\infty(\mathbb R)}+ \|1\|_{L^\infty(\mathbb R)}= 
    \|g(u)\|_{L^\infty(\mathbb R)}+ \|1\|_{L^\infty(\mathbb R)} \leq 2.
\end{align}
The inequality \eqref{1demb} thus holds with $c=2$. 
\\ When $A$ is an arbitrary Young function, then by scaling and normalization, one can produce an equivalent Young function satisfying \eqref{apr23}. This results in an equivalent norm in $W^{1,A}(\mathbb R)$, whence \eqref{1demb}
holds also in this case, with a constant depending on $A$.
\\ Finally, if $\Omega$ is a half-line, then the extension operator to $\mathbb R$ of any function  $u \in W^{1,A}(\Omega)$ by even reflection about the endpoint of $\Omega$ maps $u$ into a function $\overline u \in W^{1,A}(\mathbb R)$. Moreover, $\|\overline u\|_{W^{1,A}(\mathbb R)}\leq 2 \|u\|_{W^{1,A}(\Omega)}$. Hence, the inequality \eqref{1demb}
 follows from its version in $\mathbb R$.   
\end{proof}

\subsection{Anisotropic spaces}

Recall that, as mentioned in the previous sections,
if the condition ~\eqref{A-0-aniso} is not  assumed, then the function  $\Phi_n$, and hence $\vartheta$ given by \eqref{theta}, are defined with $\Phi$ modified near zero, if necessary, so that ~\eqref{A-0-aniso} is fulfilled. The specific choice of this replacement does not affect the assumptions of the statements below.

\begin{lem}\label{lem:Phi-Psi}
   Let $n\geq 2$, let $\Psi, \Phi : \Rn \to [0, \infty)$ be $n$-dimensional Young functions, and let 
   $\Cf : [0, \infty) \to [0, \infty)$ be a non-decreasing function.
   Assume that $\Phi$ satisfies the condition ~\eqref{Phi-inf} and
    \begin{equation}\label{eq:jan11}
        \Psi(\xi) \leq c + \Phi\left(\frac{\xi}{\Cf(\vartheta(\xi))}\right) \quad \text{for $\xi \in \Rn$, }
    \end{equation}
    for some constant $c$, where $\vartheta$ is the function defined by 
    ~\eqref{theta}. Then, 
    \begin{equation}\label{eq:goal-3}
        \Psi(\Cf(s)\xi) \leq c + \Phi_n(s) + \Phi (\xi)\quad \text{for $s\geq 0$ and $\xi \in \Rn$.}
    \end{equation}
\end{lem}
\begin{proof}
    Since the function $\Phi_n$ is increasing, and for each $\xi$, the function $s \mapsto \Phi\Big(\frac{\xi}{\Cf(s)}\Big)$ is non-increasing, by the assumption~\eqref{eq:jan11}, we have that
\begin{align}
    \label{jan16}
    \Psi (\xi) \leq c + \Phi\Big(\frac{\xi}{\Cf(\vartheta(\xi))}\Big) \leq c + \Phi\Big(\frac{\xi}{\Cf(s)}\Big) \leq  c + \Phi_n(s)+\Phi\Big(\frac{\xi}{\Cf(s)}\Big) \quad \text{if $s\leq \vartheta(\xi)$,}
\end{align}
and
\begin{align}
\label{jan17}
    \Psi (\xi) \leq c + \Phi_n(\vartheta(\xi)) \leq c + \Phi_n(s) \leq c + \Phi_n(s) +\Phi\Big(\frac{\xi}{\Cf(s)}\Big) \quad \text{if $s\geq \vartheta(\xi)$.}
\end{align}
From~\eqref{jan16} and~\eqref{jan17} we deduce that
\begin{equation*}
    \Psi(\xi) \leq c + \Phi_n(s) + \Phi\left(\frac{\xi}{\Cf(s)}\right)\quad \text{for $s\geq 0$ and $\xi \in \Rn$.}
\end{equation*}
whence ~\eqref{eq:goal-3} follows.
\end{proof}
{
\begin{lem}
    \label{recover}
Let $n\geq 2$. Assume that $\Phi(\xi)=A(|\xi|)$ and $\Psi (\xi)= B(|\xi|)$ for some Young functions $A$ and $B$, with $A$ satisfying~\eqref{A-0} and~\eqref{A-inf}. {Let $\Cf$ be as in \eqref{jan50}.}
Then the condition \eqref{eq:assumpt1} for $\Phi$ and $\Psi$ implies the condition \eqref{eq:inq-ass2} {for $B$ and a Young function equivalent to $A$.}
 \end{lem}
\begin{proof}
 {For functions $\Phi$ and $\Psi$ as in the statement, the function $\vartheta (\xi)$ only depends on $|\xi|$. Let $\widehat \vartheta : [0, \infty) \to [0, \infty)$ be the function such that $\vartheta (\xi)= \widehat \vartheta (|\xi|)$ for $\xi \in \rn$.}
 Equation \eqref{theta} then takes the form
\begin{equation}  \label{thetaA} 
      A_n(\widehat\vartheta(t)) = A\left(\frac{t}{\Cf(\widehat\vartheta(t))}\right) \quad \text{ for $t \geq 0$.}
\end{equation}
Namely, 
$$A(H_n^{-1}(\widehat\vartheta (t)))=A\left(\frac{t}{\Cf(\widehat\vartheta(t))}\right),$$
and hence
\begin{align}
    \label{mar40}
    H_n^{-1}(\widehat\vartheta (t))\Cf(\widehat\vartheta(t))=t  \quad \text{ for $t \geq 0$.}
\end{align}
The latter identity ensures that $\widehat\vartheta$  is strictly monotone and invertible, and 
\begin{align} \label{mar41}
    \widehat\vartheta ^{-1}(s) =  H_n^{-1}(s)E(s) \quad \text{ for $s \geq 0$.}
\end{align}
{Now, the condition \eqref{eq:assumpt1} is equivalent to requiring that
 \begin{align*}
     B(t) \leq c + A\left(\frac{t}{\Cf(\widehat\vartheta(t))}\right) \quad \text{ for $t \geq 0$,}
 \end{align*}
 and the latter is in turn equivalent to
  \begin{align}
     \label{mar42}
      B(\widehat\vartheta ^{-1}(s)) \leq c+A\left(\frac{\widehat\vartheta ^{-1}(s)}{\Cf(s)}\right) \quad \text{ for $s \geq 0$,}
 \end{align}
 Coupling \eqref{mar41} with \eqref{mar42} yields:
  \begin{align}
     \label{mar43}
    B(H_n^{-1}(s)E(s)) \leq c + A(H_n^{-1}(s)) \quad \text{ for $s \geq 0$.}
 \end{align}
 Thus, there exist  constants $c'$ and $t_0$ such that 
\begin{align}\label{apr55}
    B(t\Cf(H_n(t)))\leq A(c't)
  \quad \text{ for $t \geq t_0.$}
\end{align}
Hence, the conclusion follows.}
\end{proof}}
The following lemma is a counterpart of Lemma~\ref{int-conv} for Orlicz spaces of vector-valued functions. Its proof follows along the same lines as that of 
Lemma~\ref{int-conv}, and will be omitted.
{\begin{lem}\label{int-conv-aniso}
   Let $\Omega$ be a measurable subset of $\rn$, with $n\geq 2$, and 
    let $\Phi$ be a finite-valued  non-degenerate $n$-dimensional Young function. Assume that the sequence $\{U_k\} \subset L^\Phi(\Omega, \rn)$ and the function $U\in L^\Phi(\Omega, \rn)$ satisfy
    $$\int_\Omega \Phi\Big(\frac{1}{\lambda_0}(U_k-U)\Big)\d x \xrightarrow{k \to \infty} 0$$
    for some $\lambda_0>0$. {Let $\lambda_1>0$ be such that
    $$\int_\Omega \Phi\Big(\frac{1}{\lambda_1}U\Big)\d x<\infty.$$}
 Then,  
    \begin{align}
        \label{int-conv1-aniso}
        \int_\Omega \Phi\Big(\frac{1}\lambda U_k\Big)\d x \xrightarrow{k \to \infty}  \int_\Omega \Phi\Big(\frac{1}\lambda U\Big)\d x
    \end{align}
    for every $\lambda \geq 2\max(\lambda_0, \lambda_1\}$. \\ In particular, if $\{U_k\} \subset E^{\Phi}(\Omega)$, $U \in E^{\Phi}(\Omega)$ and $\|U_k - U\|_{L^{\Phi}(\Omega, \rn)}\to 0$, then~\eqref{int-conv1-aniso} holds for every $\lambda > 0$.
\end{lem}
}

The last lemma of this section is a version of Lemma~\ref{lem:convEAn} for Orlicz spaces of vector-valued functions.
\begin{lem}\label{lem:Phin-conv}
    Let $\Omega$ be an open set in $\rn$, with $n\geq 2$. Let $\Phi : \Rn \to [0, \infty)$ be a non-degenerate $n$-dimensional Young function  satisfying the condition ~\eqref{Phi-inf}. 
    If  $|\Omega| = \infty$, assume, in addition,   that $\Phi$ also satisfies the condition \eqref{A-0-aniso}. 
    If 
    the sequence  $\{U_k\} \subset \sobzP$ and the function $U\in \sobzP$ are such that $U_k \to U$   modularly in $\sobzP$, then $\|U - U_k\|_{L^{\Phi_n}(\Omega)} \xrightarrow{k \to \infty} 0$.
\end{lem}

\begin{proof}
    The proof parallels that of Lemma~\ref{lem:convEAn}, by replacing $A$ with $\Phi$ and $A_n$ with $\Phi_n$. We limit ourselves to pointing out a few variations. The first difference arises in equation ~\eqref{eq:goal-part2-1},  which can be simplified as
    \begin{equation*}
        \int_{\{|u - u_k| \geq 2s\}} \Phi_n\left(\frac{|u - u_k|}{\lambda}\right)\d x \leq \int_{\{|u - u_k| \geq 2s\}} \Phi_n\left(\frac{2(|u - u_k| - s)}{\lambda}\right)\d x \leq \int_{\Omega} \Phi_n\left(\frac{2|g_s(u - u_k)|}{\lambda}\right)\d x.
    \end{equation*}
    Equation ~\eqref{eq:help-dec1}, can be skipped and, instead of the inequality~\eqref{eq:sobemb}, one can use the inequality~\eqref{eq:sobemb-2}, which holds for arbitrary $\Omega$. Similarly to~\eqref{eq:mar28-6}, we obtain
    \begin{align*}
        \int_{\Omega} \Phi_n\left(\frac{2|g_s(u - u_k)|}{\lambda}\right)\d x &\leq \int_{\Omega} \Phi_n\left(\frac{|g_s(u - u_k)|}{c\lambda_c\left(\int_{\Omega} \Phi\left( \frac{\nabla(g_s (u - u_k))}{\lambda_c}\right)\d y\right)^{1/n}}\right)\d x\\ &\leq \int_{\Omega} \Phi\left(\frac{\nabla(g_s (u-u_k))}{\lambda_c}\right)\d x
        \leq \int_{\Omega} \Phi\left(\frac{|\nabla u - \nabla u_k|}{\lambda_c}\right)\d x \xrightarrow{k \to \infty} 0\,.
    \end{align*}
    The rest of the arguments is completely analogous.
\end{proof}

\section{Proofs of the main results}\label{sec:proofs}
\subsection{Isotropic spaces}

\begin{proof}[Proof of Theorem~\ref{theo-conA}]

We begin by proving that, if $\nabla u \in L^{\A}(\Omega)$, then $\nabla(f(u))$ belongs to $L^{\A}(\Omega)$. Since $u$ is weakly differentiable and $f$ is Lipschitz continuous, $f(u)$ is also weakly differentiable, and $\nabla (f(u)) = f'(u)\nabla u$ a.e. in $\Omega$.
Let $C_f > 0$ be the Lipschitz constant of the function $f$. Hence, $\|f'\|_{L^\infty(\mathbb R)} \leq C_f$ and therefore
\begin{equation}\label{eq:apr1}
    \int_{\Omega} \A \left(\frac{|f'(u)\nabla u|}{\lambda} \right)\d x \leq \int_{\Omega} \A \left(\frac{C_f|\nabla u|}{\lambda} \right)\d x < \infty\,,
\end{equation}
for sufficiently large $\lambda$. Furthermore, if $u \in L^A(\Omega)$, then $f(u) \in L^A(\Omega)$. Indeed, 
\begin{equation}\label{eq:apr2}
    \int_{\Omega} \A\left(\frac{|f(u)|}\lambda\right)\d x \leq \int_{\Omega} \A\left(\frac{C_f|u| + |f(0)|}{\lambda}\right)\d x \leq \int_{\Omega} \A\left(\frac{2C_f|u|}{\lambda}\right)\d x + \int_{\Omega} \A\left(\frac{2|f(0)|}{\lambda}\right)\d x\,.
\end{equation}
As $u \in L^A(\Omega)$, we have that $\int_{\Omega} \A\left(2C_f|u|/\lambda\right)\d x<\infty$ for sufficiently large $\lambda$. Also, if $|\Omega| < \infty$, then $\int_{\Omega} \A\left(2|f(0)|/\lambda\right)\d x = |\Omega|\A\left(2|f(0)|/\lambda\right) < \infty$. If $|\Omega| = \infty$ and $f(0) = 0$, then $\int_{\Omega} \A\left(2|f(0)|/\lambda\right)\d x\, = 0$. Thereby, the rightmost side of~\eqref{eq:apr2} is finite. Altogether, we have shown that $f(u) \in \sob$ for any $u \in \sob$.
\\
We next focus
on the continuity of the operator $T_f$. Consider any sequence $\{u_k\}$, converging  modularly to some function $u$ in $\sob$. Namely, there exists $\lambda_c  \geq  \|u\|_{W^{1, A}(\Omega)}$ such that
\begin{equation}\label{eq:lambdac}
     \int_{\Omega} \A\left(\frac{|\nabla u_k - \nabla u|}{\lambda_c}\right)\d x \xrightarrow{k \to \infty} 0\, \quad \text{ and } \quad  \int_{\Omega} \A\left(\frac{|u_k - u|}{\lambda_c}\right)\d x \xrightarrow{k \to \infty} 0\,.
\end{equation}
Plainly, $f(u_k) \to f(u)$ modularly in $L^A(\Omega)$. Indeed, if $\lambda > C_f\lambda_c$, then
\begin{equation}\label{eq:mar28}
    \int_{\Omega} \A\left(\frac{|f(u_k) - f(u)|}{\lambda}\right)\d x \leq \int_{\Omega} \A\left(\frac{C_f|u_k - u|}{\lambda}\right)\d x \xrightarrow{k \to \infty} 0\,.
\end{equation}
It remains to prove that $\nabla(f(u_k)) \to \nabla(f(u))$ modularly in $L^\A(\Omega)$, namely that
\begin{equation}\label{apr40}
    \int_{\Omega} \A\left(\frac{|\nabla(f(u_k)) - \nabla(f(u))|}{\lambda}\right)\d x  \xrightarrow{k \to \infty} 0
\end{equation}
for sufficiently large $\lambda >0$.
\\
Assume, for the time being, that $\Omega$ is bounded. For any $\varepsilon > 0$ let us set
\begin{align*}
    G_0 &= \{x \in \Omega: \nabla u(x) = 0\}\,,\\
    G_{\varepsilon} &= \{x \in \Omega: \varepsilon > |\nabla u(x)| > 0\}\,,\\
    \Omega_\varepsilon &= \{x \in \Omega: |\nabla u(x)| \geq \varepsilon\}\,.
\end{align*}
The first limit in \eqref{eq:lambdac} implies that $\int_{G_0} \A( |\nabla u_k|/
\lambda_c) \xrightarrow{k \to \infty} 0$. Hence, 
\begin{equation}\label{eq:apr3}
    \int_{G_0} \A\left(\frac{|\nabla(f(u_k)) - \nabla(f(u))|}{\lambda}\right)\d x = \int_{G_0} \A\left(\frac{|f'(u_k)\nabla u_k|}{\lambda}\right)\d x \leq \int_{G_0} \A\left(\frac{C_f|\nabla u_k|}{\lambda}\right)\d x \xrightarrow{k \to \infty} 0
\end{equation}
for $\lambda > C_f\lambda_c$.
\\
Moreover, since $|G_\varepsilon| \xrightarrow{\varepsilon \to 0} 0$, from ~\eqref{eq:lambdac} we deduce that 
\begin{align}\label{eq:apr4}
    &\int_{G_{\varepsilon}} \A\left(\frac{|\nabla(f(u_k)) - \nabla(f(u))|}{\lambda}\right)\d x\\ &\leq \int_{G_{\varepsilon}} \A\left(\frac{2C_f|\nabla u_k|}{\lambda}\right)\d x + \int_{G_{\varepsilon}} \A\left(\frac{2C_f|\nabla u|}{\lambda}\right)\d x \nonumber
    \xrightarrow{k \to \infty} 2\int_{G_{\varepsilon}} \A\left(\frac{2C_f|\nabla u|}{\lambda}\right)\d x\xrightarrow{\varepsilon \to 0} 0
\end{align}
for $\lambda > 4C_f\lambda_c$. Note that the convergence in $k$  rests upon Lemma~\ref{int-conv}.
\\
Thanks to~\eqref{eq:apr3} and~\eqref{eq:apr4}, it only remains  to show that  for every $\lambda_0 > 12C_f\lambda_c$ and  $\varepsilon > 0$ one has
\begin{equation}\label{eq:goal1}
    \int_{\Omega_\varepsilon} \A\left(\frac{|f'(u_k)\nabla u_k - f'(u)\nabla u|}{\lambda_0}\right)\d x \xrightarrow{k \to \infty} 0\,.
\end{equation}
Owing to Lusin's Theorem, for any $\delta > 0$, there exists an open set
$\Udel$, and a continuous function $g_{\delta} : \R \to \R$ such that $|\Udel| < \delta$, $g_\delta = f'$ in $\mathbb R\setminus S_\delta$, and  $|g_{\delta}| \leq C_f$.
\\ Clearly,
\begin{align}\label{eq:apr5}
     &\int_{\Omega_\varepsilon} \A\left(\frac{|f'(u_k)\nabla u_k - f'(u)\nabla u|}{\lambda}\right)\d x \leq  \int_{\Omega_\varepsilon} \A\left( \frac{3|f'(u_k)\nabla u_k - g_{\delta}(u_k)\nabla u_k|}{\lambda}\right)\d x\\ &+ \int_{\Omega_\varepsilon} \A\left(\frac{3|g_{\delta}(u_k)\nabla u_k - g_{\delta}(u)\nabla u|}{\lambda}\right)\d x + \int_{\Omega_\varepsilon} \A\left( \frac{3|g_{\delta}(u)\nabla u - f'(u)\nabla u|}{\lambda}\right)\d x \nonumber
\end{align}
for $\lambda >0$.
  Owing to equation \eqref{eq:lambdac} and  the continuity of $g_\delta$,
every subsequence of $\left\{\A\left( \frac{3(g_{\delta}(u_k)\nabla u_k - g_{\delta}(u)\nabla u)}{\lambda}\right)\right\}$ has a subsequence converging to $0$  a.e.. Thus, $ \A\left(\frac{3(g_{\delta}(u_k)\nabla u_k - g_{\delta}(u)\nabla u)}{\lambda}\right) \to 0$  in measure. Moreover,  
\begin{equation*}
    \A\left(  \frac{3|g_{\delta}(u_k)\nabla u_k - g_{\delta}(u)\nabla u|}{\lambda}\right) \leq \A\left(  \frac{6C_f|\nabla u_k|}{\lambda}\right) + \A\left(  \frac{6C_f|\nabla u|}{\lambda}\right)\,, 
\end{equation*}
and the right-hand side is equiintegrable by Lemma~\ref{int-conv}, provided that $\lambda > 12C_f\lambda_c$. This piece of information, combined with Vitali Convergence Theorem, implies that 
\begin{equation}\label{eq:dec1}
    \int_{\Omega_\varepsilon} \A\left(  \frac{3|g_{\delta}(u_k)\nabla u_k - g_{\delta}(u)\nabla u|}{\lambda}\right)\d x \xrightarrow{k \to \infty} 0
\end{equation}
for $\lambda > 12C_f\lambda_c$.
\\
Next, we  have that
\begin{equation}\label{eq:apr6}
    \int_{\Omega_\varepsilon} \A\left(  \frac{3|f'(u_k)\nabla u_k - g_{\delta}(u_k)\nabla u_k|}{\lambda}\right)\d x \leq \int_{\Omega_{\varepsilon}} \chi_{\Udel}(u_k) \A\left(  \frac{6C_f|\nabla u_k|}{\lambda}\right)\d x\,.
\end{equation}
Let us set $V_{k,\varepsilon} = \Omega_{\varepsilon} \cap \{x \in \Omega: |\nabla u_k(x)| \geq \tfrac{\varepsilon}{2}\}$. As $\nabla u_k \to \nabla u$ in measure,   $$|\Omega_\varepsilon \setminus V_{k,\varepsilon}| \xrightarrow{k \to \infty} 0\,.$$ Therefore, by Lemma~\ref{int-conv} and Vitali Convergence Theorem, 
\begin{equation}\label{eq:dec3}
    \int_{\Omega_{\varepsilon} \setminus V_{k, \varepsilon}} \chi_{\Udel}(u_k) \A\left(  \frac{6C_f|\nabla u_k|}{\lambda}\right)\d x \leq \int_{\Omega_{\varepsilon} \setminus V_{k, \varepsilon}} \A\left(  \frac{6C_f|\nabla u_k|}{\lambda}\right)\d x \xrightarrow{k \to \infty} 0
\end{equation}
for $\lambda > 12C_f\lambda_c$.
Also,
\begin{equation}\label{eq:may9}
    |V_{k, \varepsilon} \cap u_k^{-1}(\Udel)| = \frac{2}{\varepsilon}\int_{V_{k, \varepsilon}}\chi_{\Udel}(u_k) \cdot \frac{\varepsilon}{2}\d x   \leq \frac{2}{\varepsilon} \int_{\Omega} \chi_{\Udel}(u_k)|\nabla u_k|\d x\,.
\end{equation}
\\ {As  a consequence of~\cite[Theorem 1]{GBM} and~\cite[Lemma 3.1]{BL}, one has that, if  
 $Z$ is any measurable subset of $\R$, then the functional \begin{equation}\label{continuity}
    W^{1, 1}(\Omega)   \ni  v \mapsto \int_{\Omega} \chi_{Z}(u)|\nabla v|\d x \quad \text{ is continuous on $W^{1, 1}(\Omega)\,$.}
    \end{equation}}
Note that as $\Omega$ is bounded, $W^{1, A}(\Omega) \to W^{1, 1}(\Omega)$, and modular convergence in $W^{1, A}(\Omega)$ implies convergence in $W^{1, 1}(\Omega)$. Hence, thanks to  the property \eqref{continuity},  
\begin{equation}\label{eq:may10-2}
    \int_{\Omega} \chi_{\Udel}(u_k)|\nabla u_k|\d x \xrightarrow{k \to \infty} \int_{\Omega} \chi_{\Udel}(u)|\nabla u|\d x = \int_{\Udel}  \mathcal{H}^{n-1}(\{u=t\}) \d t \xrightarrow{\delta \to 0} 0,
\end{equation}
where the last equality holds thanks to the coarea formula~\cite[Theorem 1.1]{Coarea}. Consequently, by~\eqref{eq:may9} and~\eqref{eq:may10-2}, 
$$\lim _{\delta \to 0}\Big(\limsup_{k} |V_{k, \varepsilon} \cap u_k^{-1}(\Udel)|\Big)=0,$$ whence
\begin{equation*}
   \lim _{\delta \to 0} \bigg(\limsup_{k\to \infty} \int_{V_{k, \varepsilon}} \chi_{\Udel}(u_k) \A\left(  \frac{6C_f|\nabla u_k|}{\lambda}\right)\d x\bigg) = 0.
\end{equation*}
Combining the latter limit
with equations ~\eqref{eq:apr6} and~\eqref{eq:dec3} yields
\begin{equation}\label{eq:dec4}
\int_{\Omega_\varepsilon} \A\left(  \frac{3|f'(u_k)\nabla u_k - g_{\delta}(u_k)\nabla u_k|}{\lambda}\right)\d x \xrightarrow{k \to \infty} 0\,.
\end{equation}
To  estimate the last integral on the right-hand side of the inequality \eqref{eq:apr5}, notice that
\begin{equation}\label{eq:apr7}
     \int_{\Omega_\varepsilon} \A\left(  \frac{3|g_{\delta}(u)\nabla u - f'(u)\nabla u|}{\lambda}\right)\d x \leq \int_{\Omega_{\varepsilon}} \chi_{\Udel}(u) \A\left(  \frac{6C_f|\nabla u|}{\lambda}\right)\d x\,.
\end{equation}
A chain analogous to ~\eqref{eq:may9} and an application of the coarea formula as in \eqref{eq:may10-2} enables showing that
 $|\Omega_{\varepsilon} \cap u^{-1}(\Udel)| \xrightarrow{\delta \to 0} 0$. Thus,
\begin{equation*}
    \int_{\Omega_\varepsilon} \A\left(  \frac{3|g_{\delta}(u)\nabla u - f'(u)\nabla u|}{\lambda}\right)\d x \xrightarrow{\delta \to 0} 0\,,
\end{equation*}
whence, via ~\eqref{eq:apr5},~\eqref{eq:dec1} and~\eqref{eq:dec4}, {we infer that
the limit \eqref{eq:goal1} holds 
for every $\lambda_0 > 12C_f\lambda_c$.}
This concludes the proof of \eqref{apr40} for bounded $\Omega$.
\\ 
It remains to remove this additional assumption. Let $\Omega$ be any open set in $\rn$. Given any bounded open set  $G \subset \Omega$, we already know that
\begin{equation*}
     \int_{G} \A\left(  \frac{|f'(u_k)\nabla u_k - f'(u)\nabla u|}{\lambda_0}\right)\d x \xrightarrow{k \to \infty} 0\,,   
\end{equation*}
for $\lambda_0 > 12C_f\lambda_c$.
Fix  any such $\lambda_0$ and  any $\varepsilon > 0$,  and choose $G$ such that $\int_{\Omega \setminus G} \A\left(  \frac{|\nabla u|}{\lambda_c}\right) < \varepsilon$. We have that
\begin{align*}
    \int_{\Omega \setminus U} \A\left(  \frac{|f'(u_k)\nabla u_k - f'(u)\nabla u|}{\lambda_0}\right)\d x &\leq \int_{\Omega \setminus U} \A\left(  \frac{|\nabla u_k|}{2\lambda_c}\right)\d x + \int_{\Omega \setminus U} \A\left(  \frac{|\nabla u|}{2\lambda_c}\right)\d x\\
    & \quad \xrightarrow{k \to \infty} 2\int_{\Omega \setminus U}\A\left(  \frac{|\nabla u|}{2\lambda_c}\right)\d x \leq 2\varepsilon\,,
\end{align*}
where the limit holds thanks to Lemma~\ref{int-conv}. Thereby,  
\begin{equation*}
    \limsup_{k \to \infty} \int_{\Omega} \A\left(  \frac{|f'(u_k)\nabla u_k - f'(u)\nabla u|}{\lambda_0}\right)\d x \leq 2\varepsilon\,.
\end{equation*}
Thanks to the arbitrariness of $\varepsilon$, we conclude that
\begin{equation*}
    \int_{\Omega} \A\left(  \frac{|f'(u_k)\nabla u_k - f'(u)\nabla u|}{\lambda_0}\right)\d x \xrightarrow{k \to \infty} 0\,,
\end{equation*}
whence \eqref{apr40} holds for $\lambda >  12C_f\lambda_c$.
\end{proof}
\begin{proof}[Proof of Theorem~\ref{1d}]
 Part (i).  Given $M > 0$, let $f_M : \mathbb R \to \mathbb R$ be the Lipschitz continuous function defined as 
\begin{align}
    \label{fM}
    f_M (t) = \begin{cases}
        f(-M) & \quad \text{if $t<-M;$}
        \\ f(t) & \quad \text{if $-M \leq t \leq M;$}
        \\ f(M) & \quad \text{if $t>M$.}
    \end{cases}
\end{align}
By~\eqref{embWR}, any function $u \in W^{1, A}(\Omega)$ also belongs to $L^{\infty}(\Omega)$. Choose $M = \|u\|_{L^{\infty}(\Omega)}$ in \eqref{fM}.  Since $f(u) = f_M(u)$,  by Theorem~\ref{theo-conA}, which also holds with the same proof for $n=1$,
applied to $f_M$, we obtain that $f(u) \in W^{1, A}(\Omega)$. Next, consider any sequence $u_k \to u$ modularly in $\sob$. In particular, sequence $\{\|u_k\|_{W^{1, A}(\Omega)}\}$ is bounded. From~\eqref{1demb}, we infer that there exists $M > 0$ such that $\|u\|_{L^{\infty}(\Omega)} \leq M$ and $\|u_k\|_{L^{\infty}(\Omega)} \leq M$ for $k \in \mathbb N$. Since $f(u) = f_M(u)$, and  $f(u_k) = f_M(u_k)$ for $k\in \mathbb N$,  an application of Theorem~\ref{theo-conA} again, with $f$ replaced with $f_M$, tells us that $f(u_k) \to f(u)$ modularly in $W^{1, A}(\Omega)$. Hence, ~\eqref{jan23bis} is established.
\\
Part (ii).  
If $f$ is Lipschitz continuous, then \eqref{jan23_0} follows from
Theorem \ref{coro-conA}, which also holds for $n=1$. 
The case when $f$ is
 just locally Lipschitz continuous can be reduced to the previous one
 via the same argument as in the proof of Part $(i)$. The choice of $M$ in \eqref{fM} now depends on the inequality $\|u\|_{L^{\infty}(\Omega)} \leq c\|u'\|_{L^A(\Omega)}$ for $c = c(A, \Omega)$ and all $u \in V_0^{1, A}(\Omega)$, which is a consequence of inequalities~\eqref{1demb0} and~\eqref{1demb0inf}.
\end{proof}
\begin{proof}[Proof of Theorem~\ref{theo:conB1}]
{Here, and in what follows, we shall repeatedly  make use  of the fact that, if 
 $u: \Omega \to \mathbb R$ is any weakly differentiable function, and $f: \mathbb R \to \mathbb R$ is a locally Lipschitz continuous function, then $f(u)$ is weakly differentiable if and only if the function $f'(u)\nabla u$ is locally integrable in $\Omega$. Moreover, if this is the case, then $\nabla (f(u))= f'(u)\nabla u$
 a.e. in $\Omega$. 
 We refer to \cite{MarcusMizel72} for this and related results.}
\\ Part (i).
Observe that as $A$ satisfies~\eqref{A-convinf}, and $\Omega \in \mathcal{G}_{1/n'}$, by \eqref{w1Ainfinity} the space $W^{1, A}(\Omega)$ is continuously embedded into $L^{\infty}(\Omega)$. Therefore,  if $u_k \to u$   modularly in $ W^{1, A}(\Omega)$, there exists $M > 0$ such that $\|u\|_{L^{\infty}(\Omega)} \leq M$ and $\|u_k\|_{L^{\infty}(\Omega)} \leq M$ for every $k$. Therefore, the same argument as in the proof of Theorem \ref{1d} yields \eqref{jan23bis}.

Part (ii). As a first step, we show that, if  $u \in \sob$, then $f(u) \in \sobB$. Our assumptions on $f$ ensure that
\begin{align}
    \label{apr41}
    |f(t)| \leq \vk|t|\Cf(\vk|t|) + |f(0)| \quad \text{for $t \in \mathbb R$.}
\end{align}
  Let $c$ be the constant appearing in Lemma~\ref{lem:lem1}. Thus,
    \begin{align}\label{eq:mar28-2}
        \int_{\Omega} B\left(\frac{|f(u)|}{4\lambda \vk} \right)\d x &\leq \int_{\Omega} B\left(\frac{|f(0)|}{2\lambda \vk}\right)\d x + \int_{\Omega} B\left(\frac{|u|\Cf(\vk|u|)}{2\lambda}\right)\d x \\
        &\leq |\Omega|B\left(\frac{|f(0)|}{2\lambda \vk}\right) + c|\Omega| + \int_{\Omega} A_n(\vk|u|)\d x + \int_{\Omega} A\left(\frac{|u|}{\lambda}\right)\d x\,.\nonumber
    \end{align}
    The rightmost side of the chain \eqref{eq:mar28-2} is finite 
    for sufficiently large $\lambda > 0$, inasmuch as $|\Omega|<\infty$, $u \in L^A(\Omega)$, and the embedding \eqref{embEsigma} holds with $\sigma =n$. This tells us that $f(u) \in L^B(\Omega)$. 
    \\
  Furthermore, by~\eqref{eq:lem1}, 
\begin{equation}\label{eq:jann1}
            \int_{\Omega} B \left( \frac{|f'(u)| |\nabla u|}{\lambda} \right)\d x \leq  \int_{\Omega} B \left( \frac{\vk \Cf(\vk|u|)|\nabla u|}{\lambda} \right)\d x \leq  \int_{\Omega} \left( c + A\left(\frac{2\vk|\nabla u|}{\lambda}\right) + A_n(\vk|u|)\right)\d x < \infty\,,
    \end{equation}
     for  $\lambda \geq 2\vk  \|u\|_{W^{1, A}(\Omega)}$, {whence $|f'(u)| |\nabla u| \in L^B$. Altogether, we have that $f(u) \in \sobB$.}

As far as the continuity of the operator $T_f$ is concerned, consider any sequence $u_k \to u$  modularly in $\sob$. Thus, we may pick $\lambda_c  \geq  \|u\|_{W^{1, A}(\Omega)}$ such that
    \begin{equation}\label{eq:mar28-lambdac}
        \int_{\Omega} A\left( \frac{|u_k - u|}{\lambda_c}\right)\d x \xrightarrow{k \to \infty} 0 \quad \text{and} \quad \int_{\Omega} A\left( \frac{|\nabla u_k - \nabla u|}{\lambda_c}\right)\d x \xrightarrow{k \to \infty} 0\,. 
    \end{equation}
    To prove that $f(u_k) \to f(u)$ modularly in $L^B$,   note that 
    \begin{align}\label{eq:apr8}
        B\left(\frac{|f(u_k(x)) - f(u(x))|}{\lambda}\right) &\leq B\left(\frac{2|f(u(x))|}{\lambda}\right) + B\left(\frac{2|f(u_k(x))|}{\lambda}\right) \\ \nonumber
        &\leq B\left(\frac{2|f(u(x))|}{\lambda}\right) + B\left(\frac{4|f(0)|}{\lambda}\right) + B\left(\frac{4\vk|u_k(x)|\Cf(\vk|u_k(x)|)}{\lambda}\right)
    \end{align}
    for  $\lambda > 0$ and $x \in \Omega$.
    By Lemma~\ref{lem:lem1}, 
    \begin{equation}\label{eq:apr9}
        B\left(\frac{4\vk|u_k(x)|\Cf(\vk|u_k(x)|)}{\lambda}\right) \leq  c + A\left(\frac{8\vk|u_k(x)|}{\lambda}\right) + A_n(\vk|u_k(x)|)\quad \text{for  $\lambda > 0$ and $x \in \Omega$.}
    \end{equation}
    From Lemmas~\ref{lem:convEAn} and \ref{int-conv} we deduce that $\int_{\Omega} A_n(\vk|u_k|)\d x \xrightarrow{k \to \infty} \int_{\Omega} A_n(\vk|u|)\d x$.  Thanks to Lemma~\ref{int-conv} again, this piece of information ensures that the right-hand side of~\eqref{eq:apr9} is equiintegrable for $\lambda > 16\vk \lambda_c$. As every subsequence of $\{f(u_k(x))\}$ has a subsequence that converges pointwise a.e. to $f(u(x))$, the sequence $\{f(u_k)\}$ converges  to $f(u)$ in measure. Thereby, the estimates~\eqref{eq:apr8} and~\eqref{eq:apr9}, via Vitali Convergence Theorem, imply that
    \begin{equation}\label{eq:mar28-3}
        B\left(\frac{|f(u_k(x)) - f(u(x))|}{\lambda}\right) \xrightarrow{k \to \infty} 0\,
    \end{equation}
    for $\lambda > 16\vk\lambda_c$. Hence, $f(u_k) \to f(u)$ modularly in $L^B(\Omega)$. 
    \\
    It remains to prove the modular convergence of $\nabla(f(u_k))$ to $\nabla (f(u))$ in $L^B(\Omega)$. Given $\varepsilon > 0$, let us set:
\begin{align*}
    G_0 &= \{x \in \Omega: \nabla u(x) = 0\}\,,\\
    G_{\varepsilon} &= \{x \in \Omega: \varepsilon > |\nabla u(x)| > 0\}\,,\\
    \Omega_\varepsilon &= \{x \in \Omega: |\nabla u(x)| \geq \varepsilon\}\,.
\end{align*}
By \eqref{eq:mar28-lambdac},  $\int_{G_0} A(|\nabla u_k|/\lambda_c) \xrightarrow{k \to \infty} 0$.  Moreover, since any subsequence of $\{u_k\}$ admits a subsequence which converges a.e. to $u$, the sequence $f'(u_k)$ is bounded a.e. in $\Omega$, and $f'(u_k)|\nabla u_k| \to 0$  {in measure} in $G_0$. By Lemma~\ref{lem:lem1},  
\begin{equation*}
    B\left(\frac{|f'(u_k)\nabla u_k|}{2\lambda_c}\right) \leq c + A\left(\frac{|\nabla u_k|}{\lambda_c}\right) + A_n(\vk|u_k|)\,.
\end{equation*}
Lemma~\ref{lem:convEAn} ensures that the right-hand side of the last inequality converges in $L^1(G_0)$. Hence, owing to Vitali Convergence Theorem,
\begin{equation}\label{eq:may6}
    \int_{G_0} B\left(\frac{|f'(u_k)\nabla u_k|}{2\lambda_c}\right)\d x \xrightarrow{k \to \infty} 0\,.
\end{equation}
 Next, note that $|G_{\varepsilon}| \xrightarrow{\varepsilon \to 0} 0$. Moreover, 
\begin{align}\label{eq:may7}
    &\int_{G_{\varepsilon}} B\left(\frac{|\nabla(f(u_k)) - \nabla(f(u))|}{\lambda}\right)\d x \\ \nonumber 
    &\leq \int_{G_{\varepsilon}} \left(c + A\left(\frac{|\nabla u_k|}{2\lambda_c}\right) + A_n(\vk|u_k|)\right)\d x + \int_{G_{\varepsilon}} B\left(\frac{|\nabla (f(u))|}{2\vk \lambda_c}\right)\d x \\ \nonumber 
    & \quad \xrightarrow{k \to \infty} \int_{G_{\varepsilon}} \left(c + A\left(\frac{|\nabla u|}{2\lambda_c}\right) + A_n(\vk|u|)\right)\d x + \int_{G_{\varepsilon}} B\left(\frac{|\nabla (f(u))|}{2\vk \lambda_c}\right)\d x \xrightarrow{\varepsilon \to 0} 0\,,
\end{align}
 for $\lambda \geq 8\lambda_c\vk$, where the limits hold thanks to
Lemma~\ref{int-conv} and equation ~\eqref{eq:jann1}.
\\ Our last task is to show that
\begin{equation}\label{eq:goal1-2}
    \int_{\Omega_\varepsilon} B\left(\frac{|f'(u_k)\nabla u_k - f'(u)\nabla u|}{\lambda}\right)\d x \xrightarrow{k \to \infty} 0
\end{equation}
for  $\lambda \geq 24\vk \lambda_c$ and  $\varepsilon > 0$.
By Lusin's Theorem, for any $\delta > 0$, there exists an open set $\Udel$ and a continuous function $g_{\delta} : \R \to \R$ such that $|\Udel| < \delta$ and $g_{\delta} = f'$ in $\R \setminus \Udel$. We may also assume that $|g_{\delta}(t)| \leq \vk \Cf(t)$, since the function 
$$\mathbb R \ni t \mapsto \min\{\vk \Cf(\vk|t|), \max\{(-\vk \Cf(\vk|t|), g_{\delta}(t))\}\}$$
 is also continuous and agree with $f'$ in $\R \setminus \Udel$, as $|f'| \leq \vk \Cf(\vk|t|)$ for $t \in \mathbb R$.
\\ Consequently,
\begin{align}\label{eq:apr5-2}
     &\int_{\Omega_\varepsilon} B\left(\frac{|f'(u_k)\nabla u_k - f'(u)\nabla u|}{\lambda}\right)\d x \leq  \int_{\Omega_\varepsilon} B\left(\frac{3|f'(u_k)\nabla u_k - g_{\delta}(u_k)\nabla u_k|}{\lambda}\right)\d x\\ \nonumber &+ \int_{\Omega_\varepsilon} B\left(\frac{3|g_{\delta}(u_k)\nabla u_k - g_{\delta}(u)\nabla u|}{\lambda}\right)\d x + \int_{\Omega_\varepsilon} B\left(\frac{3|g_{\delta}(u)\nabla u - f'(u)\nabla u|}{\lambda}\right)\d x
\end{align}
for  $\lambda > 0$.  
Every subsequence of $\Big\{B\left(\frac{3|g_{\delta}(u_k)\nabla u_k - g_{\delta}(u)\nabla u|}{\lambda}\right)\Big\}$ has a subsequence that converges to $0$  a.e. in $\Omega$. This ensures  that $ B\left(\frac{3|g_{\delta}(u_k)\nabla u_k - g_{\delta}(u)\nabla u|}{\lambda}\right) \to 0$   in measure. Moreover, 
\begin{align*}
    &B\left(\frac{3|g_{\delta}(u_k)\nabla u_k - g_{\delta}(u)\nabla u|}{\lambda}\right) \leq c + A\left(\frac{|\nabla u_k|}{2\lambda_c}\right) + A_n(\vk|u_k|) + B\left(\frac{\Cf(\vk|u|)|\nabla u|}{4\lambda_c}\right)
\end{align*}
for $\lambda \geq 24\lambda_c\vk$. Hence, via Lemma~\ref{int-conv}, equation ~\eqref{eq:jann1}, and Vitali Convergence Theorem one infers that
\begin{equation}\label{eq:apr8-2}
    \int_{\Omega_\varepsilon} B\left(\frac{3|g_{\delta}(u_k)\nabla u_k - g_{\delta}(u)\nabla u|}{\lambda}\right)\d x \xrightarrow{k \to \infty} 0\,.
\end{equation}
Next, observe that  $|\Omega_\varepsilon \cap u^{-1}(\Udel)| \xrightarrow{\delta \to 0} 0$. Indeed, the coarea formula~\cite[Theorem 1.1]{Coarea} and the fact that $|S_{\delta}| < \delta$ imply
\begin{equation*}
    |\Omega_{\varepsilon} \cap u^{-1}(S_{\delta})| \leq \frac{1}{\varepsilon} \int_{\Omega_{\varepsilon}} \chi_{\Udel}(u)|\nabla u|\d x = {\frac{1}{\varepsilon}}\int_{\Udel} \mathcal{H}^{n-1}(\{u=t\})\d t \xrightarrow{\delta \to 0} 0\,.
\end{equation*}
Thereby, 
\begin{align}\label{apr45}
    &\int_{\Omega_\varepsilon} B\left(\frac{3|g_{\delta}(u)\nabla u - f'(u)\nabla u|}{\lambda}\right)\d x = \int_{\Omega_\varepsilon}\chi_{\Udel}(u) B\left(\frac{3|g_{\delta}(u)\nabla u - f'(u)\nabla u|}{\lambda}\right)\d x  \\ \nonumber
    & \quad \leq \int_{\Omega_{\varepsilon}} \chi_{\Udel}(u) \left(c + A\left(\frac{|\nabla u|}{\lambda_c}\right) + A_n(\vk|u|)\right)\d x \xrightarrow{\delta \to 0} 0
\end{align}
for $\lambda \geq 24\lambda_c\vk$.
\\
Now, set $V_{k, \varepsilon} = \{x \in \Omega_{\varepsilon} : |\nabla u_k| \geq \frac{\varepsilon}{2}\}$. We have
\begin{equation}\label{eq:may9-2}
    |V_{k, \varepsilon} \cap u_k^{-1}(\Udel)| = \frac{2}{\varepsilon} \int_{V_{k, \varepsilon}} \chi_{\Udel}(u_k)\frac{\varepsilon}{2}\d x  \leq \frac{2}{\varepsilon} \int_{\Omega_{}} \chi_{\Udel}(u_k)|\nabla u_k|\d x\,.
\end{equation}
By the property ~\eqref{continuity}, the operator $v \mapsto \int_{\Omega} \chi_{\Udel}(v) |\nabla v|\d x$ is continuous on $W^{1, 1}(\Omega)$. Therefore,  
\begin{equation}\label{eq:may10}
    \int_{\Omega} \chi_{\Udel}(u_k)|\nabla u_k|\d x \xrightarrow{k \to \infty} \int_{\Omega} \chi_{\Udel}(u)|\nabla u|\d x = \int_{\Udel}H^{n-1}(\{u=t\})\, \d t\xrightarrow{\delta \to 0} 0\,.
\end{equation}
Coupling ~\eqref{eq:may9-2} with ~\eqref{eq:may10} tells us that  ${\lim_{k}} |V_{k, \varepsilon} \cap u_k^{-1}(\Udel)| \xrightarrow{\delta \to 0} 0$. Thus, 
\begin{align}\label{eq:may11}
    &\limsup_{k} \int_{V_{k , \varepsilon}} B\left(\frac{3|f'(u_k)\nabla u_k - g_{\delta}(u_k)\nabla u_k|}{\lambda}\right)\d x \nonumber\\
    &\leq \limsup_{k} \int_{V_{k, \varepsilon}} \chi_{\Udel}(u_k)\left(c + A\left(\frac{|\nabla u_k|}{2\lambda_c}\right) + A_n(\vk|u_k|)\right)\d x \xrightarrow{\delta \to 0} 0
\end{align}
for $\lambda \geq 24\lambda_c\vk$, where
the limit in $\delta$ relies upon the uniform integrability of $A\left(\frac{|\nabla u_k|}{2\lambda_c}\right)$ and $ A_n(\vk|u_k|)$.

Since $\nabla u_k \to \nabla u$ in measure, we have that $|\Omega_{\varepsilon} \setminus V_{k, \varepsilon}| \xrightarrow{k \to \infty} 0$.  By  uniform integrability, this implies
\begin{align}\label{eq:may12}
    \int_{\Omega_{\varepsilon} \setminus V_{k , \varepsilon}} & B\left(\frac{3|f'(u_k)\nabla u_k - g_{\delta}(u_k)\nabla u_k|}{\lambda}\right)\d x \\ \nonumber
    &\leq \int_{\Omega_{\varepsilon} \setminus V_{k, \varepsilon}} 2\left(c + A\left(\frac{|\nabla u_k|}{2\lambda_c}\right) + A_n(\vk|u_k|)\right)\d x \xrightarrow{k \to \infty} 0
\end{align}
for $\lambda \geq 24\lambda_c\vk$.
\\
Equation \eqref{eq:goal1-2} follows  from \eqref{eq:apr5-2}, \eqref{eq:apr8-2}, \eqref{apr45}, and \eqref{eq:may12}.
\\ Thanks to the arbitrariness of $\varepsilon$, from equations \eqref{eq:may6}
, \eqref{eq:may7}, and \eqref{eq:goal1-2} we deduce that
\begin{align}
    \label{apr46}
    \int_{\Omega} B\left(\frac{|f'(u_k)\nabla u_k - f'(u)\nabla u|}{\lambda}\right)\d x \xrightarrow{k \to \infty} 0
\end{align}
for $\lambda \geq 24\lambda_c\vk$. The modular convergence  $\nabla (f(u_k))\to \nabla (f(u))$ in $L^B(\Omega)$ is thus   established. The proof of the modular convergence   $ f(u_k)\to  f(u) $ in $\sobB$ is complete.
\end{proof}

{\begin{proof}[Proof Theorem~\ref{theo:conB2}]
 Part (i).
 Thanks to the  inequality \eqref{feb4inf}, this case can be dealt with via the same argument as in the proof of Part (i) of Theorem~\ref{theo:conB1}.
 \\ Part (ii). Since the assumptions imposed on the function $\Df$ only concern its behavior near zero, given $\overline t >0$ we may assume, by modifying, if necessary, $\Df$ far from zero, that 
$\Df(t) \leq A_n(t)$ for  $t\geq \overline t$. Therefore,  by coupling the assumptions ~\eqref{eq:inq-assD}  and ~\eqref{eq:inq-ass2}, we may suppose that
\begin{equation*}
    A^{-1}(t)\Cf(\Df^{-1}(t)) \leq B^{-1}(t) \quad \text{ for  $t \geq 0$.}
\end{equation*}
The assumptions of Lemma~\ref{lem:inqD} are thus satisfied.
\\ We begin by 
 proving that $f(u) \in W^{1, B}(\Omega)$ for any $u \in W^{1, A}(\Omega)$. Lemma~\ref{lem:inqD} entails that
\begin{equation}\label{eq:feb1}
    \int_{\Omega} B\left(\frac{|f(u)|}{\lambda}\right)\d x \leq \int_{\Omega} B\left(\frac{\vk|u|\Cf(\vk|u|)}{\lambda}\right)\d x \leq \int_{\Omega} \left(\Df(\vk|u|) + A\left(\frac{\vk|u|}{\lambda}\right) \right)\d x
\end{equation}
for $\lambda >0$. The last integral is finite for $\lambda \geq \vk \|u\|_{W^{1, A}(\Omega)}$, since  $u \in W^{1, A}(\Omega)$, and, as a consequence of  Lemma~\ref{lem:convED}, $u \in E^\Df(\Omega)$.
This shows  that $f(u) \in L^{B}(\Omega)$. 
\\ Similarly,  
\begin{equation*}
    \int_{\Omega} B\left(\frac{|f'(u)\nabla u|}{\lambda}\right)\d x \leq \int_{\Omega} B\left(\frac{\vk|\nabla u|\Cf(\vk|u|)}{\lambda}\right)\d x \leq \int_{\Omega} \left( \Df(\vk|u|) + A\left(\frac{\vk|\nabla u|}{\lambda}\right) \right)\d x
\end{equation*}
for $\lambda >0$. As  the last integral is   finite 
for $\lambda \geq \vk \|u\|_{W^{1, A}(\Omega)}$, we have that $\nabla(f(u)) \in L^{B}(\Omega)$. In conclusion,  $f(u) \in W^{1, B}(\Omega)$.  
\\ 
Let us now establish the continuity of the operator $T_f$ from $W^{1, A}(\Omega)$ into $W^{1, B}(\Omega)$. Assume that  $\{u_k\} \subset \sob$ and $u \in W^{1, A}{(\Omega)}$ are such that $u_k \to u$  modularly in $W^{1, A}(\Omega)$. Choose $\lambda_c  \geq \|u\|_{W^{1, A}(\Omega)}$ such that
\begin{equation*}
        \int_{\Omega} A\left( \frac{|u_k - u|}{\lambda_c}\right)\d x \xrightarrow{k \to \infty} 0 \quad \text{and} \quad \int_{\Omega} A\left( \frac{|\nabla u_k - \nabla u|}{\lambda_c}\right)\d x \xrightarrow{k \to \infty} 0\,. 
\end{equation*}
 Consider any open set $G \subset \Omega$ which is a finite union of domains from the class $\mathcal{G}_{1/n'}$, bounded connected Lipschitz domains, for instance. Such a set $G$ can be chosen in such a way that $|\Omega \setminus G|$ is arbitrarily small. As~\eqref{eq:inq-ass2} is in force,
 Theorem~\ref{theo:conB1} and Remark \ref{rem:domains} ensure that
\begin{equation}\label{eq:help3}
    \int_{G} B\left(\frac{|f(u_k) - f(u)|}{\lambda}\right) \xrightarrow{k \to \infty} 0
\end{equation}
for $\lambda \geq 24\lambda_c\kappa$. Next, notice that
\begin{align}\label{mar15}
    \int_{\Omega \setminus G} B\left(\frac{|f(u_k) - f(u)|}{\lambda}\right) \d x&\leq \int_{\Omega \setminus G} \left(\Df(\vk|u_k|) + A\left(\frac{|u_k|}{2\lambda_c}\right) + B\left(\frac{|f(u)|}{\vk \lambda_c}\right)\right)\d x\\ \nonumber \quad 
    &\xrightarrow{k \to \infty} \int_{\Omega \setminus G} \left(\Df(\vk|u|) + A\left(\frac{|u|}{2\lambda_c}\right) + B\left(\frac{|f(u)|}{\vk \lambda_c}\right)\right)\d x\,,
\end{align}
where the limit holds owing to Lemmas~\ref{int-conv} and~\ref{lem:convED} for convergence. The integrand in the last integral in \eqref{mar15} belongs to $L^1(\Omega)$. As a consequence, for any $\varepsilon > 0$ there exists $G$ as above such that this integral does not exceed $\varepsilon$. By the arbitrariness of $\varepsilon$ and~\eqref{eq:help3}, we conclude that $\int_{\Omega} B(|f(u_k) - f(u)|/\lambda)\d x \xrightarrow{k \to \infty} 0$, whence $f(u_k) \to f(u)$ modularly in $\sobB$. 
\\
It remains to show the modular convergence of $\nabla(f(u_k))$ to $\nabla(f(u))$ in $L^B(\Omega)$.  Let $G$  be a set as above. Thanks to the assumption ~\eqref{eq:inq-ass2},
 Theorem~\ref{theo:conB1} and Remark \ref{rem:domains} imply that
\begin{equation}\label{eq:help4}
    \int_{G} B\left(\frac{|f'(u_k)\nabla u_k - f'(u)\nabla u|}{\lambda}\right) \xrightarrow{k \to \infty} 0
\end{equation}
 for every $\lambda \geq 24\lambda_c\vk$.
Furthermore,
\begin{align*}
    \int_{\Omega \setminus G} B\left(\frac{|f'(u_k)\nabla u_k - f'(u)\nabla u|}{\lambda}\right) &\leq \int_{\Omega \setminus G} \left(\Df(\vk|u_k|) + A\left(\frac{|\nabla u_k|}{2\lambda_c}\right) + B\left(\frac{|f'(u)\nabla u|}{\vk \lambda_c}\right)\right)\d x\\ \quad
    &\xrightarrow{k \to \infty} \int_{\Omega \setminus G} \left(\Df(\vk|u|) + A\left(\frac{|\nabla u|}{2\lambda_c}\right) + B\left(\frac{|f'(u)\nabla u|}{\vk \lambda_c}\right)\right)\d x\,,
\end{align*}
where the convergence follows from Lemmas~\ref{int-conv} and~\ref{lem:convED}. By the same arguments as in the above proof of the modular convergence 
of $f(u_k)$ to $f(u)$ in $L^B(\Omega)$, one can establish the modular convergence of $\nabla(f(u_k))$ to $\nabla(f(u))$ in $L^B(\Omega)$, and hence of $f(u_k)$ to $f(u)$ in  $\sobB$.
\end{proof}
}
\begin{proof}[Proof of Proposition \ref{rem:norm}, sketched]
 The strengthening of the conclusions of Parts (ii) of Theorems \ref{theo:conB1} and \ref{theo:conB2} provided by this proposition amounts to the fact that if $u_k \to u$ in the modular topology of $W^{1,A}(\Omega)$, then 
$$ \int_{\Omega} B\left( \frac{|f(u_k) - f(u)|}{\lambda}\right)\d x \xrightarrow{k \to \infty} 0 \quad \text{and} \quad \int_{\Omega} B\left( \frac{|\nabla (f(u_k)) - \nabla (f(u))|}{\lambda}\right)\d x \xrightarrow{k \to \infty} 0$$
for every $\lambda >0$, and not just for sufficiently large $\lambda$. The latter restriction on $\lambda$ arises in the proof of 
Theorem~\ref{theo:conB1} from an inequality of the form
    \begin{equation}\label{eq:mar31-1}
        B\left(\frac{C\Cf(\vk|v|)|\nabla v|}{\lambda}\right) \leq c + A_n(\vk |v|) +  A\left(\frac{C|\nabla v|}{\lambda}\right),
    \end{equation}
    which is exploited several times. {The proof of Theorem~\ref{theo:conB1} uses inequality~\eqref{eq:mar31-1} with several constants $C > 0$, but never exceeding $6\vk$.
    \\ As now $E$ is Young function, given constants $C$ and $\lambda$, thanks 
   to the property~\eqref{Alambda} and Lemma~\ref{lem:lem1}, we can  make use of the inequality
    \begin{equation}\label{eq:mar31-2}
        B\left(\frac{C\Cf(\vk|v|)|\nabla v|}{\lambda}\right) \leq B\left(\frac{C\Cf(\vk\tfrac{\Lambda}{\lambda}|v|)|\nabla v|}{\Lambda}\right) \leq c + A_n \left(\vk\tfrac{\Lambda}{\lambda}|v| \right) + A\left(\frac{2C|\nabla v|}{\Lambda}\right)\,,
    \end{equation}
    where $0 < \lambda \leq \Lambda$.
     One can argue as in the proof of Theorem~\ref{theo:conB1}, starting with any $\lambda > 0$ instead of choosing a particular $\lambda$. Then, one can  take $\Lambda = 24\vk\max(\lambda, \lambda_c, \|u\|_{W^{1, A}(\Omega)})$ and replace every use of the inequality~\eqref{eq:mar31-1} with the inequality~\eqref{eq:mar31-2}. This   affects the constants inside the functions $A$ and $A_n$ in the estimates. Nonetheless, the choice of $\Lambda$ shall guarantee the equiintegrability of the families $\left\{ A_n \left(\kappa \frac{\Lambda|u_k|}{\lambda}\right) \right\}$ and $\left\{A \left( \frac{2C|\nabla u_k|}{\Lambda}\right) \right\}$ for any $C \leq 6\kappa$, via Lemmas~\ref{lem:convEAn} and~\ref{int-conv}, respectively. The rest of the proof follows along the same lines.
    \\
    Similarly, if in the proof of Theorem~\ref{theo:conB2}, one replaces the inequality
    \begin{equation*}
        B\left(\frac{C\Cf(\vk|v|)|\nabla v|}{\lambda}\right) \leq \Df(\vk|v|) + A\left(\frac{C|\nabla v|}{\lambda}\right)
    \end{equation*}
    with
    \begin{equation*}
        B\left(\frac{C\Cf(\vk|v|)|\nabla v|}{\lambda}\right) \leq \Df\left(\frac{\vk\Lambda|v|}{\lambda}\right) + A\left(\frac{C|\nabla v|}{\Lambda}\right)\,,
    \end{equation*}
    then one obtains the desired conclusion.}
\end{proof}
{
\begin{proof}[Proof of Corollary \ref{coro:An}]
  The conclusion is a consequence of Theorem \ref{theo:conB2}  and Lemma \ref{LemmaF=An}.   
\end{proof}}

\subsection{Anisotropic spaces}\label{sec:proofs-aniso}
{\color{black}

The outline of the proofs of the results of Section~\ref{sec:aniso} is essentially the same as that of the proofs accomplished in the previous section. We thus limit ourselves to providing sketches of the relevant proofs and outlining the necessary modifications.

\begin{proof}[Proof of Theorem~\ref{theo-conPhi}] 
Part $(i)$.
The proof  follows along the same lines as the proof of Theorem~\ref{theo-conA}, with $A$ replaced by $\Phi$.   The inequalities involved in the proof hold thanks to the convexity of $\Phi$ and the monotonicity of a function $t \mapsto \Phi(t\xi)$ for any $\xi \in \Rn$. Since membership in the space $V_0^{1, \Phi}(\Omega)$ only depends on the gradient of  functions, the lines~\eqref{eq:apr2} and~\eqref{eq:mar28} have to be skipped, whereas $\lambda_c$ in~\eqref{eq:lambdac} must be chosen so that $\lambda_c \geq \|u\|_{V^{1, \Phi}_0(\Omega)}$ and
\begin{equation*}
    \int_{\Omega} \Phi \left( \frac{\nabla u_k - \nabla u}{\lambda_c} \right) \d x \xrightarrow{k \to \infty} 0\,.
\end{equation*}
The embedding
 $V_0^{1, \Phi}(\Omega)\to W^{1, 1}(\Omega)$ if $|\Omega|<\infty$ also plays a role in the proof.
\\
Part $(ii)$. The embedding \eqref{apr60}
ensures that any sequence $\{u_k\}$ converging modularly in $V^{1, \Phi}_0(\Omega)$ satisfies the bound $\|u_k\|_{L^{\infty}(\Omega)} \leq M$ for some $M > 0$ and every $k$. This fact allows one to replace the function $f$ with a Lipschitz continuous function $f_M$, defined as in~\eqref{fM}, and make use of Part $(i)$.
    
\end{proof}}
The proof of Theorem~\ref{theo:conPhi0} is similar to those of Theorems~\ref{theo:conB1} and~\ref{theo:conB2}. It relies upon Lemmas~\ref{lem:Phi-Psi} and~\ref{lem:Phin-conv}, instead of their isotropic versions Lemmas~\ref{lem:lem1} and~\ref{lem:convEAn}. 
\begin{proof}[Proof of Theorem~\ref{theo:conPhi0}]
Part  $(i)$. The argument is the same  as in the proof of Part $(ii)$ of Theorem~\ref{theo:conB1}. One has simply to replace $A$ with $\Phi$,  $A_n$ with $\Phi_n$, $B$ with $\Psi$, Lemma~\ref{lem:lem1} with Lemma~\ref{lem:Phi-Psi}, and Lemma~\ref{lem:convEAn} with Lemma~\ref{lem:Phin-conv}. As the spaces $V_0^{1, \Phi}(\Omega)$ and $V_0^{1, \Psi}(\Omega)$ are only defined via the gradient of  functions, equations ~\eqref{eq:mar28-2} and~\eqref{eq:apr8}-\eqref{eq:mar28-3} have to be skipped. Moreover, $\lambda_c$ in~\eqref{eq:mar28-lambdac} has to chosen so that $\lambda_c \geq \|u\|_{V_0^{1, \Phi}(\Omega)}$ and
    \begin{equation*}
        \int_{\Omega} \Phi \left( \frac{\nabla u_k - \nabla u}{\lambda_c} \right) \d x \xrightarrow{k \to \infty} 0\,.
    \end{equation*}
    \\ Part  $(ii)$. By Part  $(i)$, given any bounded open set $G \subseteq \Omega$ and any $\lambda \geq 24\lambda_c \vk$ we have that
    \begin{equation}\label{eq:mar28-4}
    \int_{U} \Psi\left(\frac{f'(u_k)\nabla u_k - f'(u)\nabla u}{\lambda}\right) \xrightarrow{k \to \infty} 0\,.
\end{equation}
On the other hand, by Lemma~\ref{lem:Phi-Psi} with $c = 0$, 
 \begin{align}\label{mar28-5}
    \int_{\Omega \setminus G} \Psi\left(\frac{f'(u_k)\nabla u_k - f'(u)\nabla u}{\lambda}\right) \leq \int_{\Omega \setminus G} \left(\Phi_n(\vk|u_k|) + \Phi\left(\frac{\nabla u_k}{2\lambda_c}\right) + \Phi_n(\vk|u|) + \Phi\left(\frac{\nabla u}{2\lambda_c}\right)\right)&\d x \nonumber\\
    \xrightarrow{k \to \infty} 2\int_{\Omega \setminus G} \left(\Phi_n(\vk|u|) + \Phi\left(\frac{\nabla u}{2\lambda_c}\right)\right)&\d x\,,
\end{align}
where the limit follows from Lemmas~\ref{int-conv} and~\ref{lem:Phin-conv}. The set $G$ can be chosen so that the right-hand side of~\ref{mar28-5} is arbitrarily small. Coupling this fact with~\eqref{eq:mar28-4} yields the conclusion.
\end{proof}
\begin{proof}[Proof of Corollary \ref{coro:ortho}]
   Part $(i)$. Using the assumption ~\eqref{eq:ass-ortho} and  arguing as in Lemma~\ref{lem:lem1} enable us to deduce that there exists a constant $c$ such that 
    \begin{equation*}
        B_i(\Cf(s)t/2) \leq c + \Phi_n(s) + A_i(t) \quad \text{for  $s, t \geq 0$.}
    \end{equation*}
    Hence,
    \begin{equation*}
        \Psi(\Cf(s)\xi/2) = \sum_{i=1}^{n} B_i(\Cf(s)|\xi_i|/2) \leq \sum_{i=1}^{n} \left(c + \Phi_n(s) + A_i(|\xi_i|)\right) = nc + n\Phi_n(s) + \Phi(\xi)
    \end{equation*}
    for  $\xi \in \Rn$ and  $s \geq 0$.
    Therefore,
    \begin{equation*}
        \Psi(\xi/2) \leq nc + n\Phi_n(s) + \Phi\Big(\frac{\xi}{\Cf(s)}\Big)\,,
    \end{equation*}
    whence, in particular,
        \begin{equation*}
        \Psi(\xi/2) \leq nc + n\Phi_n(\vartheta(\xi)) + \Phi\left(\frac{\xi}{\Cf(\vartheta(\xi))}\right) = nc + (n+1)\Phi\left(\frac{\xi}{\Cf(\vartheta(\xi))}\right).
    \end{equation*}
   In conclusion, we have shown that
    the condition ~\eqref{eq:assumpt1} holds with $\Psi (\xi)$ replaced with the equivalent $n$-dimensional Young function
    $\frac{1}{n+1}\Psi(\xi/2)$.
The conclusion follows via an application of Theorem~\ref{theo:conPhi0}.
\\ Part $(ii)$. The proof is analogous to that of Part $(i)$. One has just to choose  $c = 0$ in the equations above.
\\ The assertion concerning the case when ~\eqref{apr51} is in force is a straightforward consequence of Theorem~\ref{theo-conPhi}, Part $(ii)$.
\end{proof}

\bigskip{}{}
{

 \par\noindent {\bf Data availability statement.} Data sharing not applicable to this article as no datasets were generated or analysed during the current study.

\section*{Compliance with Ethical Standards}\label{conflicts}

\smallskip
\par\noindent
{\bf Funding}. This research was partly funded by:
\\
(i) Polish National Science Center, grant number 2019/34/E/ST1/00120 (M. Borowski);
\\
(ii)
 Initiative of Excellence at the University of Warsaw, grant number not available (M. Borowski);
\\ (iii) GNAMPA   of the Italian INdAM - National Institute of High Mathematics (grant number not available)  (A. Cianchi);
\\ (iv) Research Project   of the Italian Ministry of Education, University and
Research (MIUR) Prin 2017 ``Direct and inverse problems for partial differential equations: theoretical aspects and applications'',
grant number 201758MTR2 (A. Cianchi);
\\ (v) Research Project   of the Italian Ministry of Education, University and
Research (MIUR) Prin 2022 ``Partial differential equations and related geometric-functional inequalities'',
grant number 20229M52AS, cofunded by PNRR (A. Cianchi).

\bigskip
\par\noindent
{\bf Conflict of Interest}. The authors declare that they have no conflict of interest.}

\printbibliography
\end{document}